\documentclass[12pt]{article}
\usepackage{amsthm}
\usepackage{amsfonts}
\usepackage{mathrsfs}
\usepackage{amsmath}
\usepackage{galois}
\usepackage[numbers,sort&compress]{natbib}
\usepackage{dsfont}
\usepackage{color}
\usepackage[perpage,symbol*]{footmisc}
\newtheorem{theorem}{Theorem}[section]
\newtheorem{proposition}[theorem]{Proposition}
\newtheorem{lemma}[theorem]{Lemma}
\newtheorem{remark}[theorem]{Remark}
\newtheorem{corollary}[theorem]{Corollary}
\newtheorem{definition}[theorem]{Definition}

\numberwithin{equation}{section}
\allowdisplaybreaks[4]
\begin{document}
\date{}
\pagestyle{plain}
\title{Distribution Flows Associated with Positivity Preserving Coercive Forms }
\author{Xian Chen$^1$, \   Zhi-Ming Ma$^2$\thanks{The corresponding author.} ,\  Xue Peng$^3$\thanks{Chen's email:  chenxian@amss.ac.cn; Ma's email: mazm@amt.ac.cn; Peng's email: pengxuemath@scu.edu.cn}
  \ \\ $^1$Xiamen University, Xiamen, 361005, China\\
\   $^2$AMSS, University of Chinese Academy of Sciences, Beijing, 100190, China\\
$^3$Sichuan University, Chengdu, 610064, China}

\maketitle \underline{}

\begin{abstract}
For a given quasi-regular positivity preserving coercive form, we construct a family of ($\sigma$-finite) distribution flows associated with the semigroup of the form. The canonical cadlag process equipped with the distribution flows behaves like a strong Markov process. Moreover, employing distribution flows we can construct optional measures and establish Revuz correspondence between additive functionals and smooth measures. The results obtained in this paper will enable us to perform a kind of stochastic analysis related to positivity preserving coercive forms.
\end{abstract}

\begin{footnotetext}
{  $^2$ Supported by National Center for Mathematics and Interdisciplinary Sciences (NCMIS), and NSFC project (11526214).
}
\end{footnotetext}

\vskip 0.2 in \noindent{\bf Keywords.}  positivity preserving coercive form;  $h$-associated process;
distribution flow; strong Markov property; positive continuous additive functional; Revuz correspondence; optional measure.

\vskip 0.2 in \noindent {\bf Mathematics Subject Classification.} Primary 31C25, 47D03; 
 Secondary 60J40, 60J45, 31C15.
\section{Introduction}
A positivity preserving coercive form is a coercive closed form with which the associated semigroup $(T_{t})_{t\geq0}$ is positivity preserving, that is, $f\geq 0$ implies $T_tf\geq 0.$ Positivity preserving semigroups as well as positivity preserving coercive forms appear in various research of mathematics and physics, and have been intensively studied by several authors. As early as in the initial of 50s of the last century, W. Feller has studied positivity preserving semigroups (\cite{Fe52, Fe54}). In 70s of the last century, A. Klein and L.J. Landau studied standard positivity preserving
semigroups via path space techniques (\cite{KL75}), and B. Simon studied positivity preserving semigroups arising from mathematical physics (\cite{Si73,Si77,Si79}). The study of positivity preserving coercive forms may be traced back to A. Beurling and J. Deny (\cite{BD58}). Afterwards J. Bliedtner (\cite{Bl71}) and A. Ancona (\cite{An74, An76}) studied intensively the subject of positivity preserving coercive forms. For more recent literature concerning positivity preserving semigroups and coercive forms,
see e.g. \cite{Sc99},\cite{Hi00},\cite{Ar08},\cite{GMNO15},\cite{De15},\cite{De16}, \cite{BCR16}, and the references therein.

When a positivity preserving semigroup $(T_{t})_{t\geq0}$ is Markovian (which is referred to as ``standard" in \cite{KL75}), or more generally sub-Markovian, i.e. $T_t1\leq 1,$ one may study it via path space techniques. Under some additional conditions, e.g. Feller property, quasi-regularity, etc.,  one may associate a nice Markov process with the underlying sub-Markovian semigroup, and thus can study the semigroup by means of stochastic analysis. When $(T_{t})_{t\geq0}$ is merely positivity preserving but not sub-Markovian, one can not directly associate with it a Markov process. In \cite{MR95}, the authors extended and completed the previous work by $h$-associating a nice Markov process with a (not necessarily sub-Markovian) positivity preserving coercive form. They implemented $h$-transformation with a strictly $\alpha$-excessive function $h$ to transfer the underlying form into a semi-Dirichlet form, and proved that a positivity preserving coercive form is $h$-associated with a nice Markov process if and only if the $h$-transform of the form is quasi-regular in the sense of \cite{MR92, MOR95}. In \cite{HMS11} the authors developed further the work of \cite{MR95} by showing that it is possible to $h\hat{h}$-associating with a pair of nice Markov processes for a quasi-regular positivity preserving coercive form, by implementing simultaneously $h$-transform with an $\alpha$-excessive function $h$ and $\hat{h}$-transform with an $\alpha$-coexcessive function $\hat{h}$.

It is evident that $h$-associated process depends on $h$ and hence a positivity preserving coercive form may have many different $h$-associated Markov processes. Inspired by the work of pseudo Hunt processes introduced in \cite{Oshima13}, in this paper we shall construct a family of ($\sigma$-finite) distribution flows on path space, associated with a given positivity preserving coercive form. The family of distribution flows is independent of the choice of $h$, and the canonical cadlag process equipped with the distribution flows behaves like a strong Markov process. Therefore  we can perform a kind of stochastic analysis directly related to a positivity preserving coercive form. In details we obtain the following results in this paper.

Let $(\mathcal{E},D(\mathcal{E}))$ be a quasi-regular positivity preserving coercive form on $L^2(E;m),$ where $E$ is a metrizable Lusin space. Let $(\mathcal{F}^0_t)_{t\geq 0}$ be the natural filtration generated by the
$E$-valued cadlag coordinate process. In Section \ref{section-distributionFlow} we define a $\sigma$-finite measure $\overline{Q}_{x,t}$ on $\mathcal{F}^0_t$ for each $t\geq 0$ and $x\in E$ via inverse $h$-transformation. We call $\overline{Q}_{x,t}$ a ($\sigma$-finite) distribution up to time $t$ and call  $(\overline{Q}_{x,t})_{t\geq 0}$ a ($\sigma$-finite) distribution flow associated with $(\mathcal{E},D(\mathcal{E}))$ (see Definition \ref{def-Qtx}). Then we prove that $\overline{Q}_{x,t}$ is generated by the family of quasi-continuous kernels of $T_s, 0\leq s\leq t$ (see Lemma \ref{lem:distribution}), and  $\overline{Q}_{x,t}$ is independent of the choice of the $\alpha$-excessive function $h$ (see Theorem \ref{thm-independent}). More over, we show that  $\overline{Q}_{x,t}$ is in general different from $\overline{Q}_{x,s}$ if $s\neq t,$ and there is in general no single measure on $\mathcal{F}^0_{\infty}$ such that its restriction to $\mathcal{F}^0_t$ is $\overline{Q}_{x,t}$ (see Theorem \ref{thm-feature}). Nevertheless, the canonical process $(X_t)_{t\geq 0}$ equipped with $(\mathcal{F}^0_t)_{t\geq 0}$ and the family of distribution flows $(\overline{Q}_{x,t})_{t\geq 0}, ~x \in E,$ enjoys also Markov property (see Theorem \ref{thm-Pseudo-Markov}).

 For deriving a strong Markov property of $(X_t)_{t\geq 0}$ in the environment of distribution flows, in Section \ref{secaugmentfiltration} we augment $(\mathcal{F}^0_t)_{t\geq 0}$ to obtain a filtration $(\mathcal{M}_t)_{t\geq 0}$ which is right continuous and universally measurable, and hence suitable to accommodate stopping times. Because there is no single measure on $\mathcal{F}^0_{\infty}$ which could be used for completion, hence the procedure of the augmentation is new and nontrivial (see (\ref{ali-Mmut}) and the proof of Theorem \ref{thm-Mu right continuous}).

 In Section \ref{Sect:strongMarkov} we study stopping times and strong Markov property related to the filtration $\{\mathcal{M}_t\}$. Denote by $\mathcal{T}$ all the $\{\mathcal{M}_t\}$-stopping times. Expanding the definition of $\overline{Q}_{x, t},$ we define a ($\sigma$-finite) measure $\overline{Q}_{x,\sigma}$ on $\mathcal{M}_{\sigma}$ for each $\sigma \in \mathcal{T}$ and $x\in E.$   We call $(\overline{Q}_{x,\sigma})_{\sigma \in \mathcal{T}}$ an  expanded ($\sigma$-finite) distribution flow associated with $(\mathcal{E},D(\mathcal{E}))$ (see Definition \ref{def-exdistribution} and Proposition \ref{pro-exdistribution}). Then we are able to show that $(X_t)_{t\geq 0}$ equipped with the filtration $(\mathcal{M}_t)_{t\geq 0}$ and the family of expanded distribution flows
$(\overline{Q}_{x,\sigma})_{\sigma \in \mathcal{T}},~x\in E,$ enjoys strong Markov property (see Theorem \ref{thm-strong pseudo Markov}).

 In the area of Dirichlet forms, positive continuous additive functionals (PCAFs), together with the Revuz correspondence between PCAFs and smooth measures, constitute an active subject and play important roles in stochastic analysis. In Section \ref{Sec-PCAF} we derive an analogy in the context of positivity preserving coercive forms. To handle the problem that there is no single measure on the path space, in Subsection \ref{Sec-PCAF1} we introduce a notion of $\mathcal{O}$-measurable positive continuous additive functionals ($\mathcal{O}$-PCAFs), in which the defining set is an optional set rather than a set in $\mathcal{M}_{\infty}$ (see Definition \ref{def-PCAF2}). $\mathcal{O}$-PCAFs in the framework of positivity preserving coercive forms play a similar role as PCAFs do in the framework of Dirichlet forms. In particular, we establish a one to one correspondence between $\mathcal{O}$-PCAFs and smooth measures (Revuz correspondence) in the subsequent subsections. In Subsection \ref{Sec-PCAF2} we establish the Revuz  correspondence by means of $h$-associated processes (see Theorem \ref{thm-revuz cor}). Then in Subsection \ref{independence} we prove that the Revuz correspondence is in fact independent of the $\alpha$-excessive function $h$ (see Corollary \ref{Revuzindependence}). To this end we introduce optional measure $\mathbb{Q}_{x}^{A}(\cdot)$ generated by an $\mathcal{O}$-PCAF $A$. Employing the structure of optional $\sigma$-field and the structure of predictable $\sigma$-field, we proof that $\mathbb{Q}_{x}^{A}(\cdot)$ is independent of the the $\alpha$-excessive function $h$ (see Theorem \ref{Qindependenth} and its proof). We believe that the concept of optional measure $\mathbb{Q}_{x}^{A}(\cdot)$ will have interest by its own and will be useful in the further study of stochastic analysis related to positivity preserving coercive forms.

In Section \ref{section: Preliminaries} below we review some results concerning $h$-associated processes as preliminaries.

\section{$h$-associated Processes}
\label{section: Preliminaries}

As a preparation and also for the convenience of the readers, in this section we review some previous results concerning $h$-associated processes of positivity preserving coercive forms.
Let $E$ be a metrizable Lusin space with its Borel $\sigma$-algebra ${\cal B}(E),$ and $m$ a $\sigma$-finite positive measure on $(E,\mathcal{B}(E)).$  We denote by $(,)_{m}$ the inner product of the (real) Hilbert space $L^2(E;m).$  For a bilinear form  $\mathcal{E}$  with domain $D(\mathcal{E})$ in $L^2(E;m),$  we write $\mathcal{E}_{\alpha}(u,v):=\mathcal{E}(u,v)+\alpha(u,v)_{m}$ for $\alpha>0.$ Write  $\hat{\mathcal{E}}(u,v):=\mathcal{E}(v,u),$ $\tilde{\mathcal{E}}(u,v):=\frac{1}{2}(\mathcal{E}(u,v)+\mathcal{E}(v,u))$, and $\check{\mathcal{E}}(u,v):=\frac{1}{2}(\mathcal{E}(u,v)-\mathcal{E}(v,u))$  for  $u,v\in D(\mathcal{E})$. Recall that a bilinear form
$(\mathcal{E},D(\mathcal{E}))$  is called a coercive closed form if its domain $D(\mathcal{E})$ is dense in $L^2(E;m)$ and the following conditions (i) and (ii) are satisfied.

(i) $(\tilde{\mathcal{E}},D(\mathcal{E}))$ is non-negative definite and closed on $L^2(E;m)$.

(ii) (Sector condition). There exists a constant $K>0$ (called continuity constant) such that
$|\mathcal{E}_1(u,v)|\leq K\mathcal{E}_1(u,u)^{\frac{1}{2}}\mathcal{E}_1(v,v)^{\frac{1}{2}}$ for all $u,v\in D(\mathcal{E})$.

We adopt the following definition from \citep{MR95}.
\begin{definition}
(cf. \cite[Definition 1.1]{MR95})
A coercive closed form $(\mathcal{E},D(\mathcal{E}))$ on $L^2(E;m)$ is called a positivity preserving coercive form , if for all $u\in D(\mathcal{E})$, it holds that $u^+\in D(\mathcal{E})$ and $\mathcal{E}(u,u^+)\geq 0$.
\end{definition}
  A semi-Dirichlet form on $L^2(E;m)$ is always a positivity preserving coercive form (cf. \cite[Remark 1.4 (iii) ]{MR95}). Recall that a coercive closed form $(\mathcal{E},D(\mathcal{E}))$ is called a semi-Dirichlet form  if for all $u\in D(\mathcal{E})$, it holds that $u^+\wedge 1\in D(\mathcal{E})$ and $\mathcal{E}(u+u^+\wedge 1,u-u^+\wedge1)\geq 0$.

For a positivity preserving coercive form $(\mathcal{E},D(\mathcal{E})),$ we denote
by $(T_{t})_{t\geq0}$ and $(G_{\alpha})_{\alpha>0}$ (resp.
$(\hat{T}_{t})_{t\geq0}$ and $(\hat{G}_{\alpha})_{\alpha>0}$)
the semigroup and resolvent (resp. co-semigroup and co-resolvent)
associated with $({\cal E},D({\cal E}))$. It is known that
a coercive form $(\mathcal{E},D(\mathcal{E}))$ is positivity preserving if and only if $(T_{t})_{t\geq0}$ and $(G_{\alpha})_{\alpha>0}$
 are  positivity preserving, that is, $f\geq 0$ implies $T_tf\geq 0$ and $ \alpha G_{\alpha}f\geq 0$  for all $t> 0,\alpha>0$  (cf. \cite [Theorem 1.5, Theorem 1.7] {MR95}). $(\mathcal{E},D(\mathcal{E}))$ is a semi-Dirichlet form if and only if $(T_{t})_{t\geq0}$ and $(G_{\alpha})_{\alpha>0}$ are sub-Markovian, in the sense that  $0\leq f\leq 1$ implies $0\leq T_tf,~ \alpha G_{\alpha}f\leq 1$ for all $t,\alpha>0$ (cf. e.g. \cite[Section I.4]{MR92} or \cite[Section 1.1]{Oshima13}).

  Let $\alpha\ge 0.$
 Recall that a function $h\in L^2(E;m)$ is called $\alpha$-excessive (resp. $\alpha$-coexcessive)
    if $h\ge 0$ and $e^{-\alpha t}T_th\leq h$ (resp. $e^{-\alpha t}\hat{T}_th\leq h$) for all $t>0$.  Given a strictly positive  $\alpha$-excessive function $h$, the conventional $h$-transform of $(T_t)_{t\geq0}$ is defined as
 \begin{equation}\label{h^2}
 T_t^hf:=h^{-1}e^{-\alpha t}T_t(hf), ~~\forall ~ t\geq0, ~ f\in L^{2}(E;h^2\cdot m).
 \end{equation}
By the discussion of \cite{MR95}, we know that $(T^h_t)_{t\geq0}$ is the semigroup associated with the semi-Dirichlet form $(\mathcal{E}^h_{\alpha},D(\mathcal{E}^h))$ on $L^{2}(E;h^2\cdot m),$ where $(\mathcal{E}^h,D(\mathcal{E}^h))$ is the $h$-transform of $(\mathcal{E},D(\mathcal{E}))$ (\cite[Definition 3.1]{MR95}) and $\mathcal{E}^h_{\alpha}(u,v)=\mathcal{E}^h(u,v)+ \alpha (u,v)_{h^2\cdot m}$  for $u,v \in D(\mathcal{E}^h).$ The following remark can be checked directly.
\begin{remark}\label{lemma-coexcessive by h trans}
Let $\gamma\geq\alpha$, then $g\in D(\mathcal{E})$ is a $\gamma$-excessive function (resp. $\gamma$-coexcessive function) of $(\mathcal{E},D(\mathcal{E}))$ if and only if
$\frac{g}{h}\in D(\mathcal{E}^h)$ is a $(\gamma-\alpha)$-excessive function (resp. $(\gamma-\alpha)$-coexcessive function) of $(\mathcal{E}^h,D(\mathcal{E}^h))$.
\end{remark}

In the remainder of this section we assume that $(\mathcal{E},D(\mathcal{E}))$ is a quasi-regular positivity preserving coercive form in the sense of \cite[Definition 4.9]{MR95}. Then every element of $D(\mathcal{E})$ admits an $\mathcal{E}$-quasi-continuous $m$-version. Below for the involved terminologies we refer to e.g. \cite{MR92} and \cite{MR95}. Let
$\mathbf{M}:=(\Omega,\mathcal{F},(\mathcal{F}_t)_{t\geq0},(X_t)_{t\geq0},
(P_x)_{x\in E_\Delta})$ be a right process with state space $E$ and transition semigroup $(P_t)_{t\geq0}.$ Here and henceforth $E_{\Delta}:= E \cup \{\Delta\}$ where $\Delta$ is an extra point of $E$ serving as  the cemetery of the process. We shall always make the convention that $f(\Delta)=0$ for any function $f$ originally defined on $E$. Let $\alpha>0.$ By the quasi-regularity we can take a strictly positive $\mathcal{E}$-quasi-continuous $\alpha$-excessive function  $h \in D(\mathcal{E})$.  We write
\begin{equation}\label{Qt}
Q_tf(x):= h(x)e^{\alpha t}P_t(h^{-1}f)(x)=h(x)E_{x}[\frac{e^{\alpha t}f(X_t)}{h(X_t)}; t<\zeta]
\end{equation}
provided the right hand side makes sense. We shall sometimes use
$Q_t(x,\cdot)$ to denote the kernel determined by (\ref{Qt}). The concept of $h$-associated process was introduced in \cite{MR95} which we restate in the definition below.
\begin{definition}\label{h-associate}
(cf. \cite[Definition 5.1]{MR95})
 Let $Q_tf$ be defined by (\ref{Qt}).

(i) We say that $(\mathcal{E},D(\mathcal{E}))$ is $h$-associated with $\mathbf{M},$ or $\mathbf{M}$ is an $h$-associated process of $(\mathcal{E},D(\mathcal{E})),$ if $Q_tf$ is an $m$-version of $T_tf$ for any $f\in L^2(E;m).$

(ii) We say that $(\mathcal{E},D(\mathcal{E}))$ is properly $h$-associated
with $\mathbf{M},$  or $\mathbf{M}$ is a properly $h$-associated process of $(\mathcal{E},D(\mathcal{E})),$  if in addition $Q_tf$ is an $\mathcal{E}$-quasi-continuous version of $T_tf$ for any $f\in L^2(E;m).$
\end{definition}
The following result was proved in \cite{MR95}.
\begin{proposition}\label{h-process} \cite[Theorem 5.2]{MR95}  Let $h \in D(\mathcal{E})$ be a strictly positive $\alpha$-excessive function. Then $(\mathcal{E},D(\mathcal{E}))$ is properly $h$-associated with an $m$-special standard process ${\bf M}$ if and only if $(\mathcal{E},D(\mathcal{E}))$ is quasi-regular. In this case $h$ is always $\mathcal{E}$-quasi-continuous.
\end{proposition}

\begin{remark}
Let $h \in D(\mathcal{E})$ be a strictly positive $\alpha$-excessive function.
If $(T^h_t)_{t\geq0}$ is associated with a right process ${\bf M}$, then by a result of \cite{Fi01}, $(\mathcal{E}^h,D(\mathcal{E}^h))$ is quasi-regular and ${\bf M}$  is in fact an m-tight special standard process properly associated with $(\mathcal{E}^h,D(\mathcal{E}^h)).$ If in addition $h$ is  $\mathcal{E}^h$-quasi-continuous, then $(\mathcal{E},D(\mathcal{E}))$ is quasi-regular and is properly $h$-associated with ${\bf M}.$

\end{remark}

Note that the $h$-associated process mentioned above depends on the $\alpha$-excessive function $h$ and hence a positivity preserving coercive form may have many different $h$-associated Markov processes. Inspired by the work of pseudo Hunt processes introduced in \cite{Oshima13}, below we shall construct a family of ($\sigma$-finite) distribution flows on path space, which is independent of $h$ and can accommodate stochastic analysis for  positivity preserving coercive forms.

\section{Distribution flows associated with $(\mathcal{E},D(\mathcal{E}))$}
\label{section-distributionFlow}

Throughout this section we assume that $(\mathcal{E},D(\mathcal{E}))$ is a quasi-regular positivity preserving coercive form  and $(T_{t})_{t\geq0}$ is its associated semigroup. From now on we fix an $\alpha>0$ and denote by $\mathcal{H}$ the totality of strictly positive $\mathcal{E}$-quasi-continuous $\alpha$-excessive functions in $D(\mathcal{E})$. Note that by the quasi-regularity of $(\mathcal{E}, D(\mathcal{E}))$ (cf.\cite[Definition 4.9]{MR95}), $\mathcal{H}$ is non-empty. By Proposition \ref{h-process} we know  that  $(T_{t})_{t\geq0}$ admits a kernel $Q_t(x,\cdot)$ determined by (\ref{Qt}) with the help of some $h\in \mathcal{H}$. Denote by
$\Omega$ the Skorohod space over $E$ with cemetery $\Delta$.
We shall make use of the kernel $Q_t(x,\cdot)$ to construct a flow of $\sigma$-finite distributions $(\overline{Q}_{x,t})_{t \geq 0}$ on $\Omega$. To this end, in what follows we fix an $h\in \mathcal{H}.$  Let
$$
\mathbf{M}^h:=(\Omega,\mathcal{F}^h,(\mathcal{F}_t^h)_{t\geq0},(X_t)_{t\geq0},
(P_x^h)_{x\in E_\Delta})
$$
with transition semigroup $(P_t^h)_{t\geq0}$ and life time $\zeta$ be a special standard process properly $h$-associated with $(\mathcal{E},D(\mathcal{E}))$. Without loss of generality we may assume that $\Omega$ is the Skorohod space over $E$ with cemetery $\Delta$, and $(X_t)_{t\geq0}$ is the canonical process  on $\Omega$, i.e., $X_t(\omega)=\omega_t$ for $\omega\in \Omega$. Denote by $(\mathcal{F}^0_t)_{t\geq 0}$ the natural filtration of $(X_t)_{t\geq 0}$ without augmentation. That is,
\begin{equation}\label{eq-natural-filtration}
\mathcal{F}^0_t:=\sigma(X_s,s\in[0,t]),~~\mathcal{F}^0_\infty:=
\sigma(X_s,s\geq 0).
\end{equation}
In this paper for any measurable space $(\Omega,\mathcal{F}),$
we shall denote by $p\mathcal{F}$ all the nonnegative
$\mathcal{F}$-measurable functions on $\Omega,$  and by $b\mathcal{F}$   all the bounded $\mathcal{F}$-measurable functions on $\Omega.$

\begin{definition}\label{def-Qtx}
 For $t\geq0, x\in E,$  we define a measure $\overline{Q}_{x,t}$ on $\mathcal{F}^0_t$ by setting
\begin{equation}\label{eq-distribution}
\overline{Q}_{x,t}(\Lambda):= \overline{Q}_{x,t}(\Lambda; t<\zeta )=h(x)E^{h}_{x}(\frac{e^{\alpha t}I_{\Lambda}}{h(X_t)};t<\zeta), ~ \forall ~ \Lambda\in\mathcal{F}^0_t,
\end{equation}
where $I_{\Lambda}$ is the indicator function of $\Lambda$, $E^{h}_{x}$ is the expectation related to $P^h_x$.
We call $\overline{Q}_{x,t}$ a {\bf \it $\sigma$-finite distribution up to time t} (in short, distribution up to t), and call $(\overline{Q}_{x,t})_{t\geq 0}$ a {\bf \it $\sigma$-finite distribution flow (in short, distribution flow) associated with $(\mathcal{E}, D(\mathcal{E}))$}.
\end{definition}

The reader should be aware that unlike usual distribution, $\overline{Q}_{x,t}$ is in general not a probability measure, or even not a finite measure, but it is a $\sigma$-finite measure (cf. (\ref{eq: Q_tf}) and (\ref{eq: Q_ttf}) below ).  This is the reason that we address ``$\sigma$-finite" in the definition. We are grateful to Mu-Fa Chen who suggested us to distinct $\overline{Q}_{x,t}$ from the usual distribution clearly.

We shall denote by $\overline{Q}_{x,t}[\cdot],$ or $\overline{Q}_{x,t}[\cdot~; t<\zeta],$ the integral related to $\overline{Q}_{x,t}$. With the convention that $f(\Delta)=0$ for any function originally defined on $E,$ comparing (\ref{eq-distribution}) with (\ref{Qt}), we see that for any $f\in p\mathcal{B}(E)$, it holds that
\begin{equation}\label{eq: Q_tf}
\overline{Q}_{x,t}[f(X_t)]= Q_tf(x).
\end{equation}
In fact we have the following extension of (\ref{eq: Q_tf}).
\begin{lemma}\label{lem:distribution}
For all $0\leq t_1<t_2<\dots<t_n=t,f\in p\mathcal{B}(E^n),~ x\in E$,~it holds that
\begin{align}\label{eq: Q_ttf}
  \overline{Q}_{x,t}&[f(X_{t_1},\dots,X_{t_n})]\\
=&\int_{E}Q_{t_1}(x,dx_1)\dots\int_{E}Q_{t_n-t_{n-1}}(x_{n-1},dx_n)
f(x_1,\dots,x_n).\nonumber
\end{align}
\end{lemma}
\begin{proof}
Let $0\leq t_1<t_2<\dots<t_n=t,f\in\mathcal{B}(E^n)^+,~q.e.~ x\in E$, we have
\begin{align*}
&\overline{Q}_{x,t}[f(X_{t_1},\dots,X_{t_n})]\\
&=h(x)E^{h}_x[\frac{e^{\alpha t_n}f(X_{t_1},\dots,X_{t_n})}{h(X_{t_n})};t_n<\zeta]\\
&=h(x)e^{\alpha t_n}\int_EP^{h}_{t_1}(x,dx_1)\dots\int_EP^{h}_{t_n-t_{n-1}}(x_{n-1},dx_n)
\frac{f(x_1,\dots,x_n)}{h(x_n)}\\
&=h(x)e^{\alpha t_1}\int_E\frac{1}{h(x_1)}P^{h}_{t_1}(x,dx_1)h(x_1)\dots \\
&\cdot h(x_{n-1})e^{\alpha(t_n-t_{n-1})}\int_EP^{h}_{t_n-t_{n-1}}(x_{n-1},dx_n)
\frac{f(x_1,\dots,x_n)}{h(x_n)}\\
&=\int_EQ_{t_1}(x,dx_1)\dots\int_EQ_{t_n-t_{n-1}}(x_{n-1},dx_{n})
f(x_1,\dots,x_n).
\end{align*}
\end{proof}
We remark that for a given  $h \in \mathcal{H},$ there are many different special standard processes $\mathbf{M}^h$ properly $h$-associated with $(\mathcal{E},D(\mathcal{E}))$. But these $h$-associated processes are equivalent to each other in the sense of \cite[IV.Definition 6.3]{MR92} (cf. also \cite[Theorem 4.2.8]{FOT11}).
Accordingly, we introduce the notion of equivalence for distribution flows as the definition below.
\begin{definition}\label{equivalent}
Let $\overline{Q}_{x,t}$ and $\overline{Q}_{x,t}'$ be two $\sigma$-finite distribution flows (may be constructed with different $h$ and/or different  $\mathbf{M}^h$). We say that $\overline{Q}_{x,t}$ and $\overline{Q}'_{x,t}$ are equivalent to each other, if there exists a Borel set $S\subset E$ such that $(\overline{Q}_{x,t})_{t\geq 0}$ and $(\overline{Q}_{x,t}')_{t\geq 0}$ are identical for all $x\in S$ and $E\setminus S$ is $\mathcal{E}$-exceptional.
\end{definition}

\begin{proposition}
Given $h\in \mathcal{H}.$ Let $\overline{Q}_{x,t}$ and $\overline{Q}_{x,t}'$ be distribution flows constructed by (\ref{eq-distribution}) with two different $h$-associated processes. Then $\overline{Q}_{x,t}$ and $\overline{Q}_{x,t}'$ are equivalent to each other in the sense of Definition \ref{equivalent}.
\end{proposition}
\begin{proof}
The proposition can be proved by employing  \cite[IV.Theorem 6.4]{MR92} or \cite[Theorem 4.2.8]{FOT11}.
\end{proof}
The above proposition can be strengthened as the theorem below.
\begin{theorem}\label{thm-independent}
$\overline{Q}_{x,t}$ defined by (\ref{eq-distribution}) is independent of $h\in \mathcal{H}.$ More precisely, let $\overline{Q}_{x,t}$ and $\overline{Q}_{x,t}'$ be distribution flows constructed by (\ref{eq-distribution}) with two different $h\in \mathcal{H}$ and $h'\in \mathcal{H}$. Then $\overline{Q}_{x,t}$ and $\overline{Q}_{x,t}'$ are equivalent to each other in the sense of Definition \ref{equivalent}.
\end{theorem}
The proof of Theorem \ref{thm-independent} needs some preparations.

Recall that for a right process $\mathbf{M}=(\Omega,\mathcal{F},(\mathcal{F}_t)_{t\geq0},(X_t)_{t\geq0},
(P_x)_{x\in E_\Delta})$ with life time  $\zeta$, a set $S\in\mathcal{B}(E)$ is called $\mathbf{M}$-invariant, if there exists $\Omega_{E\setminus S}\in\mathcal{F}$ such that
$$\Omega_{E\setminus S}\supset\{\overline{X^t_0}\cap(E\setminus S)\not=\emptyset ~\mbox{for some}~ 0\leq t<\zeta\}$$
and $P_z(\Omega_{E\setminus S})=0$ for all $z\in S$. Here, $\overline{X^t_0}(\omega)$ stands for the closure of $\{X_s(\omega)~|~s\in[0,t]\}$ in $E$ (cf. e.g. \cite[IV.Dfinition 6.1]{MR92}).
For any subset $S\subset E,$ we use $\mathcal{B}(S)^+$ to denote all the
 nonnegative Borel functions on $S$.

Let $h\in \mathcal{H}$ and $h'\in \mathcal{H}.$ Below we write $h_1=h$ and
$h_2=h'$ for simplicity. For $i=1,2,$ let
$\mathbf{M}^{h_i}:=(\Omega,\mathcal{F}^{h_i},(\mathcal{F}_t^{h_i})_{t\geq0},
(X_t)_{t\geq0}, (P_x^{h_i})_{x\in E_\Delta})$
with transition semigroup $(P_t^{h_i})_{t\geq0}$ and life time $\zeta$ be a special standard process properly $h_i$-associated with $(\mathcal{E},D(\mathcal{E}))$.
\begin{lemma}\label{invariantSet} Let $Q^{(h_i)}_t$ be defined by (\ref{Qt}) with respect to $P^{h_i}_t,~i=1,2.$ Then there exists an $\mathcal{E}$-nest $(F_k)_{k\geq 1}$ consisting of compact sets such that $h_i \in C(\{F_k\})$ for both $i=1,2$, and a Borel set $S \subset \bigcup_{k\geq 1} F_k$
 satisfying the following properties.\\
(i) $E\setminus S$ is $\mathcal{E}$-exceptional (hence $m(E\setminus S)=0$).\\
(ii) $S$ is both $\mathbf{M}^{h_1}$-invariant and $\mathbf{M}^{h_2}$-invariant.\\
(iii) $Q^{(h_1)}_t(x, E\setminus S)=Q^{(h_2)}_t(x, E\setminus S)=0$ and  $Q^{(h_1)}_tf(x)=Q^{(h_2)}_tf(x)$, for all $x\in S$, $t>0$, and $f\in\mathcal{B}(E)^+$.
\end{lemma}
\begin{proof}
Note first that by \cite[Proposition 4.2]{MR95}, $(F_k)_{k\geq 1}$ is an
$\mathcal{E}$-nest if and only if it is an $\mathcal{E}^{h_i}$-nest for $i=1,2,$  and a subset $N\subset E$ is an $\mathcal{E}$-exceptional set if and only if it is an $\mathcal{E}^{h_i}$-exceptional set for $i=1,2.$  By the property of quasi-regularity, we can take an $\mathcal{E}$-nest $(F_k)_{k\geq 1}$ consisting of compact sets and a countable family of bounded functions $\mathcal{U}\subset C(\{F_k\}),$ such that $h_i \in C(\{F_k\})$ for both $i=1,2,$ and such that $\mathcal{U}\subset \cup _{k\geq 1} D(\mathcal{E})_{F_k}$ and $\mathcal{U}$ separates the points of $E$ excepting an $\mathcal{E}$-exceptional set $N_0$.  Similar to the argument of  \cite[IV.Proposition 5.30 (ii)]{MR92}, we can show that for each $u \in \mathcal{U}$ there is an $\mathcal{E}$-exceptional set $N_u$ such that $Q^{(h_1)}_tu(z)$ and $Q^{(h_2)}_tu(z)$ are right continuous in $t$ for $z \in E\setminus N_u.$ On the other hand, by Definition \ref{h-associate} (ii) one can check that for each $t > 0$ and $u \in \mathcal{U},$ it holds that $Q^{(h_1)}_tu = Q^{(h_2)}_tu$~$\mathcal{E}$-q.e.. Therefore, by first considering rational $t,$ and then taking right limit along $t,$ we can find an $\mathcal{E}$-exceptional set $N$ such that $Q^{(h_1)}_t u(z) = Q^{(h_2)}_t u(z)$ for all $u \in \mathcal{U},t > 0$ and $z \in E\setminus N.$ This implies that $Q^{(h_1)}_t(z,\cdot)$ and $Q^{(h_2)}_t(z,\cdot)$ as kernels are identical for $t > 0$ and $z \in E_0$ where $E_0:= E \setminus (N\bigcup N_0).$ Hence, $Q^{(h_1)}_t f(z) = Q^{(h_2)}_t f(z)$ for all $z\in E_0,$ $t>0$, and $f\in p\mathcal{B}(E_0)$. Employing a standard argument (cf e.g. \cite[IV.Corollary 6.5]{MR92} and \cite[Theorem 4.1.1]{FOT11}), we can construct a Borel set $S \subset E_0$ such that $S \subset \bigcup_{k\geq 1} F_k$, $E \setminus S$ is $\mathcal{E}$-exceptional,  and $S$ is both $\mathbf{M}^{h_1}$-invariant and $\mathbf{M}^{h_2}$-invariant. By the property of invariant set and the fact that $S \subset E_0$ we can verify (iii) of the lemma.
\end{proof}
\begin{proof} (proof of Theorem \ref{thm-independent})

Let $\overline{Q}_{x,t}$ and $\overline{Q}_{x,t}'$ be distribution flows constructed by (\ref{eq-distribution}) with two different $h\in \mathcal{H}$ and $h'\in \mathcal{H}$. By Lemma \ref{lem:distribution} we see that $\overline{Q}_{x,t}$ is determined by the kernel $Q_t(x,\cdot)$ and $\overline{Q}_{x,t}'$ is determined by the kernel $Q_t'(x,\cdot)$ where $Q_t'(x,\cdot)$ is specified by (\ref{Qt}) with $h$ replaced by $h'$. Writing $h=h_1$ and $h'=h_2$, applying Lemma \ref{invariantSet} we see that there exists a Borel set $S \subset E$ such that $E\setminus S$ is $\mathcal{E}$-exceptional, and $(\overline{Q}_{x,t}')_{t\geq 0}$ and $(\overline{Q}_{x,t})_{t\geq 0}$ are identical for all $x\in S.$ Therefore   $\overline{Q}_{x,t}$ and $\overline{Q}_{x,t}'$ are equivalent to each other in the sense of Definition \ref{equivalent}.
\end{proof}

The theorem below explores more features of $\overline{Q}_{x,t}$.

\begin{theorem}\label{thm-feature}

(i) Let $s<t.$ The restriction of $\overline{Q}_{x,t}$ on $\mathcal{F}^0_s\cap \{s<\zeta\}$ is in general different from $\overline{Q}_{x,s}.$

(ii) There is in general no measure $\overline{Q}_{x}$ on $\mathcal{F}^0_\infty$ with the property that the restriction of $\overline{Q}_{x}$ on $\mathcal{F}^0_t\cap \{t<\zeta\}$ is equal to $\overline{Q}_{x,t}.$
\end{theorem}
\begin{proof}
(i) We construct an example to show this.  For $t>s,$ let $A_{ts}=\{x\in E: \hat{Q}_{s}\hat{h}(x)> 0 ~\mbox{and }~Q_{t-s}(x, E)> 1\},$ where $\hat{h}$ is some strictly positive  quasi-continuous $\alpha$-coexcessive function and $\hat{Q}_{t}\hat{h}(x)$ is defined by (\ref{Qt}) with the help of $\hat{h}$ and an $\hat{h}$-associated process of $(\hat{\mathcal{E}},D(\mathcal{E})).$  Assume that $(Q_t)_{t\geq 0}$ being a version of $(T_t)_{t\geq 0}$ is not sub-Markovian, then we can find $t>s$ such that  $m(A_{ts})>0.$  Take a compact set $F_k$ with $m(A_{ts}\cap F_k)>0$ such that $h$ is strictly positive and bounded on $F_k.$  Set $\Gamma=X^{-1}_s(A_{ts}\cap F_k)\in \mathcal{F}^0_s$, by (\ref{eq: Q_tf}) we get
\begin{align*}
\overline{Q}_{x,s}(\Gamma)=\overline{Q}_{x,s}[I_{A_{ts}\cap F_k}(X_s)]
=\int_E Q_s(x,dx_1)I_{A_{ts}\cap F_k}(x_1).
\end{align*}
One can check that $\overline{Q}_{x,s}(\Gamma)<\infty$ for all $x\in E.$ More over, if we set $B=\{x\in E:\overline{Q}_{x,s}(\Gamma)>0\},$ then we will  have $m(B)>0.$ Then for $x\in B$ by Lemma \ref{lem:distribution} we get
\begin{align}\label{Qt>Qs}
\overline{Q}_{x,t}&(\Gamma; s<\zeta)=\overline{Q}_{x,t}[I_{A_{ts}\cap F_k}(X_s)]=\overline{Q}_{x,t}[I_{A_{ts}\cap F_k}(X_s)I_E(X_t)]
\\
&=\int_E Q_s(x,dx_1)I_{A\cap F_k}(x_1) \int_EQ_{t-s}(x_1,dx_2)> \overline{Q}_{x,s}(\Gamma).\nonumber
\end{align}
Therefore, the restriction of $\overline{Q}_{x,t}$ on $\mathcal{F}^0_s\cap \{s<\zeta\}$ is in general different from $\overline{Q}_{x,s}.$

(ii) Suppose that there was a measure $\overline{Q}_{x}$ on $\mathcal{F}^0_\infty$ with the property specified as in (ii). Let $s<t$ and $\Gamma$ be as in the proof of (i).  Since $\{s<\zeta \}\supset \{t<\zeta \},$ we would have
$$ \overline{Q}_{x}(\Gamma;s<\zeta)\geq \overline{Q}_{x}(\Gamma;t<\zeta).
$$
But by (\ref{Qt>Qs}) we would have
$$\overline{Q}_{x}(\Gamma;s<\zeta)=\overline{Q}_{x,s}(\Gamma)<~ \overline{Q}_{x,t}(\Gamma;s<\zeta)= \overline{Q}_{x,t}(\Gamma;t<\zeta)= \overline{Q}_{x}(\Gamma;t<\zeta).
$$
This contradiction verifies the assertion (ii).
\end{proof}

\begin{remark}\label{remark-different from Oshima}
  The measure defined by (\ref{eq-distribution}) was first introduced in
 \cite[(3.3.11)]{Oshima13} for the dual of a semi-Dirichlet form, and was denoted by $\widehat{\mathbb{P}}_x$ with no subscript $t$.  The author did not notice that his definition of (3.3.11) is in fact dependent on $t$, and there is in general no single measure $\widehat{\mathbb{P}}_x$ satisfying his definition (3.3.11) simultaneously for different $t$.

\end{remark}

Although there is in general no single measure satisfying (\ref{eq-distribution}) simultaneously for different $t$, but  the canonical process $(X_t)_{t\geq 0}$ equipped with the whole family of distribution flows $(\overline{Q}_{x,t})_{t\geq 0}, ~x \in E,$ enjoys also Markov property. See Theorem \ref{thm-Pseudo-Markov} below. We use $(\theta_s)_{s\geq 0}$ to denote the usual time shift operators on $\Omega.$

\begin{theorem}\label{thm-Pseudo-Markov}
 Let $s\geq 0,~u\geq 0$ and $t=s+u$. Then for $\Gamma \in p\mathcal{F}^0_s$ and $Y \in p\mathcal{F}^0_{u}$,  we have
\begin{equation}\label{eq: Pseudo-Markov1}
\overline{Q}_{x,t}[\Gamma~(Y\comp\theta_s);t<\zeta]=
\overline{Q}_{x,s}[\Gamma~ \overline{Q}_{X_s,u}[Y;
u<\zeta]; s<\zeta].
\end{equation}
In particular, for $\Gamma \in p\mathcal{F}^0_s$ and $f\in p\mathcal{B}(E)$,
we have
\begin{equation}\label{eq: Pseudo-Markov2}
\overline{Q}_{x,s+u}[\Gamma~ f(X_{s+u});s+u<\zeta]=\overline{Q}_{x,s}[\Gamma ~ (Q_{u}f)(X_s);s<\zeta].
\end{equation}
\end{theorem}

\begin{proof}
Denote by $\mathcal{F}^0_{t^+}= \bigcap_{s>t}\mathcal{F}^0_s,$ then $\mathbf{M}^h$ is an $(\mathcal{F}^0_{t^+})_{t\geq0}$ Markov process.  Therefore, we have
\begin{align*}
\overline{Q}_{x,t}[\Gamma~(Y\comp\theta_s);t<\zeta]
&=h(x)E^h_x[\frac{e^{\alpha t}\Gamma~(Y\comp\theta_s)}{h(X_t)};t<\zeta]\\
&=h(x)E^h_x[e^{\alpha t}\Gamma I_{\{s<\zeta\}}E^h_x(\frac{Y~I_{(u<\zeta)}}
{h(X_u)}\comp\theta_s|\mathcal{F}^0_{s^+})]\\
&=h(x)E^h_x[e^{\alpha t}\Gamma I_{\{s<\zeta\}}E^h_{X_s}(\frac{Y}{h(X_u)};u<\zeta)]\\
&=\overline{Q}_{x,s}[\Gamma~\overline{Q}_{X_s,u}[Y;u<\zeta];s<\zeta],
\end{align*}
proving (\ref{eq: Pseudo-Markov1}). Letting $Y=f(X_u)$ we get  (\ref{eq: Pseudo-Markov2}).
\end{proof}

The corollary below explores general relations between $\overline{Q}_{x,s}$ and $\overline{Q}_{x,t}$. In (\ref{eq:generalFk}) we take an $\mathcal{E}$-nest $(F_k)_{k\geq1}$ consisting of compact sets such that $h\in C(\{F_k\})$, which ensures that the right hand side of (\ref{eq:generalFk}) is finite.  Note that when $s<t,$ by (\ref{eq:general relation 1}) it may happen that $\overline{Q}_{x,s}[\Gamma;s<\zeta]<\infty$ but $\overline{Q}_{x,t}[\Gamma;t<\zeta]=\infty$, because it may happen that $Q_{u}(x, E)=\infty$ for q.e. $x\in E$.
\begin{corollary}\label{thm-genaral-relation}
Let $s\geq 0,~u\geq 0$ and $t=s+u$. Then for $\Gamma \in p\mathcal{F}^{0}_s$ we have
\begin{equation}\label{eq:generalFk}
\overline{Q}_{x,t}[\Gamma I_{F_k}(X_t);t<\zeta]=\overline{Q}_{x,s}
[\Gamma~ Q_{u}(X_s, F_k);s<\zeta],~\forall~ k\geq 1,
\end{equation}
and
\begin{equation}\label{eq:general relation 1}
\overline{Q}_{x,t}[\Gamma;t<\zeta]=\overline{Q}_{x,s}
[\Gamma~ Q_{u}(X_s, E);s<\zeta].
\end{equation}

\end{corollary}
\begin{proof}
In (\ref{eq: Pseudo-Markov2}) taking $f=I_{F_k}$ we obtain (\ref{eq:generalFk}),
letting $f=I_E$ we obtain (\ref{eq:general relation 1}).
\end{proof}

We shall show that the process $(X_t)_{t\geq 0}$ equipped with distribution flows $(\overline{Q}_{x,t})_{t\geq 0},~x \in E$, enjoys also strong Markov property. Before showing that we should first augment the natural filtration to make it right continuous and universally measurable, and hence suitable to accommodate stopping times. Note that the distribution $P_x^h$ of the $h$-associated processes depends on $h.$  For different $h$,  the corresponding $P_x^{h}$'s might be not all equivalent ( at least it is not clear for us). Hence we can not directly make use of the augmented filtration $(\mathcal{F}_t^h)_{t\geq0}$. Thus we have to augment the natural filtration $(\mathcal{F}_t^0)_{t\geq0}$ with the $\sigma$-finite distribution flow $(\overline{Q}_{x,t})_{t\geq 0}$. This is a new and nontrivial task.

\section{Augmentation of the filtration}\label{secaugmentfiltration}
As in the previous section, in this section we fix an $h\in \mathcal{H}.$ Let
$
\mathbf{M}^{h}:=(\Omega,\mathcal{F}^h,(\mathcal{F}_t^h)_{t\geq0},(X_t)_{t\geq0},
(P_x^h)_{x\in E_\Delta})$
with transition semigroup $(P_t^h)_{t\geq0}$ and life time $\zeta$ be a special standard process properly $h$-associated with $(\mathcal{E},D(\mathcal{E}))$.   We denote by $\mathcal{B}(E_{\Delta})$ the Borel sets of $E_{\Delta}$ and by $\mathcal{P}(E_{\Delta})$ the collection of all probability measures on $\mathcal{B}(E_{\Delta})$.
For the convenience of our further discussion, from now on  we fix an  $\mathcal{E}$-nest  $(F_k)_{k\geq 1}$  consisting of compact sets such that $h \in C(\{F_k\})$ and an $\mathbf{M}^h$- invariant set $S$  constructed in Lemma \ref{invariantSet} (taking $h_1=h_2=h$).

For $\mu \in \mathcal{P}(E_{\Delta}),$ we write,
 \begin{align}
P^h_{\mu}(\Lambda):= \int_{E_{\Delta}}
P^h_{x}(\Lambda)\mu(dx), ~\forall~ \Lambda \in \mathcal{F}^h.
\end{align}

Let $(\mathcal{F}^0_t)_{t\geq0}$ be the natural filtration as defined by (\ref{eq-natural-filtration}). For $\mu \in \mathcal{P}(E_{\Delta}),$ we define for  $t\geq 0$,
\begin{align}
\overline{Q}_{\mu,t}(\Lambda):=\overline{Q}_{\mu,t}(\Lambda;t<\zeta)=
\int_{E_{\Delta}}
\overline{Q}_{x,t}(\Lambda;t<\zeta)\mu(dx), ~\forall~ \Lambda \in \mathcal{F}^0_t,
\end{align}
 where $\overline{Q}_{x,t}$ is defined by (\ref{eq-distribution}) (with the convention that $\overline{Q}_{\Delta,t}(\Lambda;t<\zeta)=0$). Note that if $\Lambda = \{X_t\in F_k, X_0\in F_k\}$, then $\overline{Q}_{\mu,t}(\Lambda)<\infty$. Hence, $\overline{Q}_{\mu,t}$ is a $\sigma$-finite measure on $\mathcal{F}^0_t.$

Let $\mathcal{F}^0_{t^+}= \bigcap_{s>t}\mathcal{F}^0_s.$  For $\mu\in\mathcal{P}(E_{\Delta})$ we define
\begin{align}\label{ali-Mmut}
\mathcal{M}^{\mu}_t:=\{\Lambda\subset\Omega~|~&\exists ~ \Lambda'\in\mathcal{F}^0_t, \,\,\Gamma\in\mathcal{F}^0_{t^+}, \,\,\mbox{s.t.} \\
&\Lambda\triangle\Lambda'\subset\Gamma,\,\,\mbox{and}~~\overline{Q}_{\mu,T}
(\Gamma;T<\zeta)=0, ~~\forall ~T>t \}.\nonumber
\end{align}
We extend $\overline{Q}_{\mu,t}$ to $\mathcal{M}^{\mu}_t$, denoted again by $\overline{Q}_{\mu,t}$,  by setting $\overline{Q}_{\mu,t}(\Lambda)=\overline{Q}_{\mu,t}(\Lambda')$ if $\Lambda$ and $\Lambda'$ are related as in (\ref{ali-Mmut}).

For our purpose of comparison, we define also
\begin{align}\label{ali-Gmu1}
\mathcal{G}^{\mu,1}_t:=\{\Lambda\subset\Omega~|~&\exists ~ \Lambda'\in\mathcal{F}^0_t, \,\,\Gamma\in\mathcal{F}^0_{t^+}, \,\,\mbox{s.t.}\\
&\Lambda\triangle\Lambda'\subset\Gamma~ \mbox{and} ~ P^h_{\mu}(\Gamma;t<\zeta)=0\}, \nonumber
\end{align}
and extend $P^h_{\mu}(\cdot; t<\zeta)$ to $\mathcal{G}^{\mu,1}_t$, denoted again by $P^h_{\mu}(\cdot; t<\zeta)$, by setting $P^h_{\mu}(\Lambda; t<\zeta):=P^h_{\mu}(\Lambda'; t<\zeta)$ if $\Lambda$ and $\Lambda'$ are related as in (\ref{ali-Gmu1}).

\begin{lemma}\label{lem-Mmut}
Let $\mathcal{M}^{\mu}_t$ and $\mathcal{G}^{\mu,1}_t$ be defined as above. Then the following assertions hold true.

(i)
$\mathcal{M}^{\mu}_t= \mathcal{G}^{\mu,1}_t$.

(ii) $\overline{Q}_{\mu,t}$ is a well defined $\sigma$-finite measure on $\mathcal{M}^{\mu}_t$, $P^h_{\mu}(\cdot; t<\zeta)$ is a well defined finite measure on
$\mathcal{M}^{\mu}_t$.

(iii) $\overline{Q}_{\mu,t}$ and $P^h_{\mu}(\cdot; t<\zeta)$ are absolutely continuous to each other.

(iv) Both $\overline{Q}_{\mu,t}$ and
 $P^h_{\mu}(\cdot; t<\zeta)$ are complete on $\mathcal{M}^{\mu}_t\bigcap \{t<\zeta\}$.
\end{lemma}
\begin{proof}
 Assertion (i) follows from the facts that for $\Gamma\in\mathcal{F}^0_{t^+}$, $P^h_{\mu}(\Gamma;t<\zeta)=0$ if and only if $P^h_{\mu}(\Gamma;T<\zeta)=0$ for all $T>t$, and $P^h_{\mu}(\Gamma;T<\zeta)=0$ if and only if $\overline{Q}_{\mu,t}(\Gamma;T<\zeta)=0$. Assertion (ii) is easy to check and we leave it to the reader. Assertion (iii) follows from the fact that $\overline{Q}_{\mu,t}(\Lambda';t<\zeta)=0$  if and only $P^h_{\mu}(\Lambda';t<\zeta)=0$ for $\Lambda'\in\mathcal{F}^0_t$. To check Assertion (iv), assume that $\Lambda \in \mathcal{G}^{\mu,1}_t$ and $P^h_{\mu}(\Lambda;t<\zeta)=0$. Then we can find $\Lambda'\in\mathcal{F}^0_t$ and $\Gamma\in\mathcal{F}^0_{t^+}$ satisfying $\Lambda\triangle\Lambda'\subset\Gamma$, $P^{h}_{\mu}(\Lambda';t<\zeta)=0$ and $P^h_{\mu}(\Gamma;t<\zeta)=0$. For any $\Lambda_0 \subset \Lambda$, we have $\Lambda_0\triangle\Lambda'\subset(\Lambda'\cup\Gamma)\in\mathcal{F}^0_{t^+}$
 and $P^h_{\mu}(\Lambda'\cup\Gamma;t<\zeta)=0$. Therefore, $\Lambda_0 \in \mathcal{G}^{\mu,1}_t$ and $P^h_{\mu}(\Lambda_0;t<\zeta)=P^{h}_{\mu}(\Lambda';t<\zeta)=0$, proving that $P^h_{\mu}(\cdot; t<\zeta)$ is complete. Employing Assertion (iii) we see that $\overline{Q}_{\mu,t}$ is also complete.
 \end{proof}
Recall that $\mathcal{H}$ is the collection of all strictly positive quasi-continuous $\alpha$-excessive functions in $(\mathcal{E}, D(\mathcal{E}))$. For any $h\in \mathcal{H},$
we denote by $(\mathcal{F}^{h,\mu}_{t})_{t\geq 0}$ the $P^h_{\mu}$-augmentation  of the filtration $(\mathcal{F}^{0}_{t})_{t\geq 0}$ in $\mathcal{F}^{h}_{\infty}$, that is,
\begin{align}\label{ali-Fhtmu}
\mathcal{F}^{h,\mu}_t:=\{\Lambda\subset\Omega~|~&\exists ~ \Lambda'\in\mathcal{F}^0_t, \,\,\Gamma\in\mathcal{F}^h_{\infty}, \,\,\mbox{s.t.}\\
&\Lambda\triangle\Lambda'\subset\Gamma~ \mbox{and} ~ P^h_{\mu}(\Gamma)=0~\}.
\nonumber
\end{align}
We use the same notation $P^h_{\mu}$ to denote its extension to $(\mathcal{F}^{h,\mu}_{t})_{t\geq 0}$.

\begin{lemma}\label{thm-relation M and F}
(i) If $\Lambda\in\mathcal{M}^{\mu}_t$, then $\Lambda\cap\{t<\zeta\}\in\mathcal{F}^{h,\mu}_t$ for any $h\in\mathcal{H}$.

(ii) For any $h\in \mathcal{H}$ we have
\begin{align}\label{eq: relation Mmu and F}
\overline{Q}_{\mu,t}&(\Lambda;t<\zeta)=E^{h}_{h\cdot \mu}[\frac{e^{\alpha t}I_{\Lambda}}{h(X_t)};t<\zeta]\\
&:=\int_{E_\Delta} h(x)E^{h}_x[\frac{e^{\alpha t}I_{\Lambda}}{h(X_t)};t<\zeta]\mu(dx), ~\forall ~ \Lambda\in\mathcal{M}^{\mu}_t. \nonumber
\end{align}
In particular, if $\mu=\delta_x$ for some  $x\in E$, then for any $h\in \mathcal{H}$ we have
\begin{equation}\label{eq: relation M and F}
\overline{Q}_{x,t}(\Lambda;t<\zeta)=h(x)E^{h}_x[\frac{e^{\alpha t}I_{\Lambda}}{h(X_t)};t<\zeta], ~\forall ~ \Lambda\in\mathcal{M}^{\mu}_t.
\end{equation}

(iii) If $s<t$, then $\mathcal{M}^{\mu}_s\subset\mathcal{M}^{\mu}_t$.
\end{lemma}
\begin{proof}
If $\Lambda\in\mathcal{M}^{\mu}_t$,  then there exist $\Lambda'\in\mathcal{F}^0_t$ and  $\Gamma\in\mathcal{F}^0_{t^+}$ such that   $\Lambda\triangle\Lambda'\subset\Gamma$ and $P^h_{\mu}(\Gamma;t<\zeta)=0$.
Then we have $[\Lambda\cap\{t<\zeta\}]\triangle[\Lambda'\cap\{t<\zeta\}]
\subset\Gamma\cap\{t<\zeta\}$  and $\Lambda'\cap\{t<\zeta\}\in\mathcal{F}^0_{t^+}$. Applying Lemma \ref{lem-Mmut} (iii) we have $P^h_{\mu}(\Gamma \cap \{t<\zeta\})=0$ for any $h\in \mathcal{H}$, consequently
$\Lambda\cap\{t<\zeta\}\in \mathcal{F}^{h,\mu}_t$ for any  $h\in \mathcal{H},$ verifying Assertion (i). Assertions (ii) is a direct consequence of (i). Assertion (iii)
can be verified directly and we omit its proof.
\end{proof}

\begin{lemma}\label{lemma-before Mt continuous}
Denote by $\mathcal{M}^{\mu}_{t^+}:=\bigcap\limits_{s>t}\mathcal{M}^{\mu}_s,$
then we have $\mathcal{M}^{\mu}_{t^+}=\mathcal{G}^{\mu,2}_t,$ where $\mathcal{G}^{\mu,2}_t$ is defined by (\ref{ali-Gmu2}) below.
\begin{align}\label{ali-Gmu2}
\mathcal{G}^{\mu,2}_t:=\{\Lambda\subset\Omega~|~&\exists ~ \Lambda'\in\mathcal{F}^0_{t^+}, \,\,\Gamma\in\mathcal{F}^0_{t^+}, \,\,\mbox{s.t.}\\
 &\Lambda\triangle\Lambda'\subset\Gamma~ \mbox{and} ~ P^h_{\mu}(\Gamma;t<\zeta)=0\}. \nonumber
\end{align}
\end{lemma}
\begin{proof}
Comparing (\ref{ali-Gmu2}) with (\ref{ali-Gmu1}), it is clear that $\mathcal{G}^{\mu,2}_t \subset \mathcal{G}^{\mu,1}_s = \mathcal{M}^{\mu}_s$
for $s>t,$ therefore $\mathcal{G}^{\mu,2}_t \subset \mathcal{M}^{\mu}_{t^+}.$
Conversely, for any $\Lambda\in\mathcal{M}^{\mu}_{t^+}$, let $s_n\downarrow t$, then $\Lambda\in\mathcal{M}^{\mu}_{s_n}$ for all $s_n$. By (\ref{ali-Gmu1}) there exist $\Lambda_n\in\mathcal{F}^0_{s_n}$ and $\Gamma_n\in\mathcal{F}^0_{s^+_n}$ such that $\Lambda\triangle\Lambda_n\subset\Gamma_n$ and $P^{h}_{\mu}(\Gamma_n;s_n<\zeta)=0$.
Let
$$\Lambda'=\bigcap\limits_{m=1}^{\infty}\bigcup\limits_{n=m}^{\infty}\Lambda_n,
~~\Gamma'=\bigcap\limits_{m=1}^{\infty}\bigcup\limits_{n=m}^{\infty}\Gamma_n,$$
then $\Lambda', \Gamma'\in\mathcal{F}^0_{t^+}$ and $\Lambda\triangle\Lambda'\subset\Gamma'$. Moreover, since $P^{h}_{\mu}(\Gamma';s_m<\zeta)\leq P^{h}_{\mu}(\bigcup^{\infty}\limits_{n=m}\Gamma_n;s_m<\zeta)=0$,
hence we have
$P^{h}_{\mu}(\Gamma';t<\zeta)=\lim\limits_{m\to+\infty}
P^{h}_{\mu}(\Gamma';s_m<\zeta)=0.$  Therefore, $\Lambda\in\mathcal{G}^{\mu,2}_t$, proving $\mathcal{M}^{\mu}_{t^+} \subset \mathcal{G}^{\mu,2}_t$.
\end{proof}
We define
\begin{align}\label{ali-Mnut all}
\mathcal{M}_t:=\bigcap\limits_{\mu\in\mathcal{P}(E_{\Delta})}
\mathcal{M}^{\mu}_t.
\end{align}
Below is a main result of this section.
\begin{theorem}\label{thm-Mu right continuous}
$(\mathcal{M}^{\mu}_t)_{t\geq0}$ is a right continuous filtration for any $\mu\in\mathcal{P}(E_{\Delta})$, and hence
$\{\mathcal{M}_t\}:= (\mathcal{M}_t)_{t\geq0}$ is a right continuous filtration.
\end{theorem}
\begin{proof}
For the first assertion, by virtue of Lemma \ref{thm-relation M and F} (iii), we need only to check that $\mathcal{M}^{\mu}_{t^+}\subset \mathcal{M}^{\mu}_t$.
Note that 
$(\mathcal{F}^{h}_t)_{t\geq0}$ is right continuous, therefore  $(X_t)_{t\geq0}$  equipped with  $(P^{h}_x)_{x\in E_{\Delta}}$ is a Markov process with respect to the filtration $(\mathcal{F}^0_{t^+})_{t\geq0}$. Hence, for any $Y\in p\mathcal{F}^0_{\infty}$, we have
$$E^{h}_{\mu}(Y|\mathcal{F}^0_{t^+})=E^{h}_{\mu}(Y|\mathcal{F}^0_t),
\,\,P^{h}_{\mu}-a.s..$$
Thus for any $\Lambda\in\mathcal{F}^0_{t^+}$, $I_{\Lambda}=E^{h}_{\mu}(I_{\Lambda}|\mathcal{F}^0_t), P^{h}_{\mu}$-$a.s.$,
which means that there exists $\Gamma\in\mathcal{F}^0_{t^+}$ such that $P^{h}_{\mu}(\Gamma)=0$ and $\{I_{\Lambda}\not=E^{h}_{\mu}(I_{\Lambda}|\mathcal{F}^0_t)\}\subset\Gamma$.
In particular we have $P^{h}_{\mu}(\Gamma;t<\zeta)=0$. Set $\Lambda' =E^{h}_{\mu}(I_{\Lambda}|\mathcal{F}^0_t)\in\mathcal{F}^0_t,$  by (\ref{ali-Gmu1}) we  get $\Lambda\in\mathcal{G}^{\mu,1}_t=\mathcal{M}^{\mu}_t$ which means $\mathcal{F}^0_{t^+}\subset\mathcal{M}^{\mu}_t$. Therefore, $\mathcal{G}^{\mu,2}_t\subset\mathcal{M}^{\mu}_t$ because by (\ref{ali-Gmu2})
$\mathcal{G}^{\mu,2}_t$ is a completion of $\mathcal{F}^0_{t^+}$ with respect to the measure $P^h_{\mu}(\cdot; t<\zeta)$. Consequently by Lemma \ref{lemma-before Mt continuous} $\mathcal{M}^{\mu}_{t^+}\subset \mathcal{M}^{\mu}_t$, proving the first assertion. The last assertion follows from the derivation below.
\begin{align*}
\mathcal{M}_t=&\bigcap\limits_{\mu\in\mathcal{P}(E_{\Delta})}
\mathcal{M}^{\mu}_t
=\bigcap\limits_{\mu\in\mathcal{P}(E_{\Delta})}\mathcal{M}^{\mu}_{t^+}
=\bigcap\limits_{\mu\in\mathcal{P}(E_{\Delta})}\bigcap\limits_{s>t}
\mathcal{M}^{\mu}_s\\
=&\bigcap\limits_{s>t}\bigcap\limits_{\mu\in\mathcal{P}(E_{\Delta})}
\mathcal{M}^{\mu}_s=\bigcap\limits_{s>t}\mathcal{M}_s=\mathcal{M}_{t^+}.
\end{align*}
\end{proof}

We shall use the notation $\mathcal{B}(E_{\Delta})^*$ to denote $\bigcap\limits_{\mu\in\mathcal{P}(E_{\Delta})}\mathcal{B}(E_{\Delta})^{\mu}$.
\begin{proposition}\label{pro-X_t measureble}
$X_t\in\mathcal{M}_t/\mathcal{B}(E_{\Delta})^*$.
\end{proposition}
\begin{proof}
For any $\mu\in\mathcal{P}(E_{\Delta})$, let $\nu(C):=P^{h}_{\mu}(X_t\in C;t<\zeta)$ for  $C\in\mathcal{B}(E_{\Delta})$,
then $\nu$ is a finite measure on $\mathcal{B}(E_{\Delta})$.
For any $A\in\mathcal{B}(E_{\Delta})^*$, there exist $A', B\in\mathcal{B}(E_{\Delta})$ such that $A\triangle A'\subset B$ and $\nu(B)=0$. Thus
$$\{X_t\in A\}\triangle\{X_t\in A'\}=X^{-1}_t(A\triangle A')\subset X^{-1}_t(B)$$
and
$$P^{h}_{\mu}(X^{-1}_t(B);t<\zeta)=P^{h}_{\mu}(X_t\in B;t<\zeta)=\nu(B)=0.$$
Since $\{X_t\in A'\}\in\mathcal{F}^0_{t}$ and $X^{-1}_t(B)\in\mathcal{F}^0_t\subset\mathcal{F}^0_{t^+}$,
we get $\{X_t\in A\}\in\mathcal{M}^{\mu}_t$. Hence, $X_t\in\mathcal{M}_t/\mathcal{B}(E_{\Delta})^*$.
\end{proof}

\section{Stopping times and strong Markov property}\label{Sect:strongMarkov}
In this section we follow the conventions and notations of the previous section. Let $(\mathcal{M}^{\mu}_{t})_{t\geq 0}$ and $(\mathcal{M}_{t})_{t\geq 0}$ be as in Theorem \ref{thm-Mu right continuous}. We set
\begin{center}
$\mathcal{M}^{\mu}_{\infty}:=\sigma(\bigcup\limits_{t\geq0}\mathcal{M}^{\mu}_t)$
~and~ $\mathcal{M}_{\infty}:=\sigma(\bigcup\limits_{t\geq0}\mathcal{M}_t)$.
\end{center}

For $B\subset E_{\Delta}$, we define the entrance time $D_B$ and the hitting time $\sigma_B$ by:
$$D_B(\omega):=\inf\{t\geq0~|~X_t(\omega)\in B\},~~ \sigma_B(\omega):=\inf\{t>0~|~X_t(\omega)\in B\}.$$

\begin{theorem}\label{thm-entrance time stopping time}
Assume that $B\in\mathcal{B}(E_{\Delta})$, then the entrance time $D_B$ and the hitting time $\sigma_B$ are $\{\mathcal{M}_t\}$-stopping times.
\end{theorem}
\begin{proof}
Let $t\geq0$, we define
$$\Phi_t:(s,\omega)\in[0,t]\times\Omega\longmapsto X_s(\omega)\in E_{\Delta}.$$
For $B\in\mathcal{B}(E_{\Delta})$, we set
\begin{center}
$A:=\{(s,\omega)\in[0,t)\times\Omega|~X_s(\omega)\in B\}=([0,t)\times\Omega)\bigcap \Phi_t^{-1}(B)$.
\end{center}
Since $(X_s)_{s\geq0}$ is $(\mathcal{F}^0_s)_{s\geq0}$-adapted and right continuous, hence we have  $A\in \mathcal{B}(\mathds{R})\times \mathcal{F}^0_t$.
Let $\Lambda=\{\omega\in\Omega~|~\exists ~s\in[0,t)~\mbox{s.t.} ~(s,\omega)\in A\}$, then $\Lambda$ is the projection of $A$ on $\Omega$. Hence  $\Lambda$ is measurable with respect to the universal completion of $\mathcal{F}^0_t$ (cf. e.g.  \cite[Proposition A.1.1]{Chen_Fukushima} ). For all  $\mu\in\mathcal{P}(E_{\Delta})$, $\mathcal{M}^{\mu}_t$ is complete with respect to the bounded measure $P^h_{\mu}(\cdot; t<\zeta)$ (cf. Lemma \ref{lem-Mmut} (iv)) , hence $\Lambda\in \mathcal{M}^{\mu}_t$. Consequently $\{D_B<t\}=\Lambda\in\mathcal{M}^{\mu}_t$ for all  $\mu\in\mathcal{P}(E_{\Delta})$, which means $\{D_B<t\}\in\mathcal{M}_t$. By the right continuity of $\{\mathcal{M}_t\}$ we see that $\{D_B\leq t\}\in\mathcal{M}_t,$ hence $D_B$ is an $\{\mathcal{M}_t\}$-stopping time. Similarly we can check that $\sigma_B$ is an $\{\mathcal{M}_t\}$-stopping time.
\end{proof}
\begin{remark}
Note that $\zeta=\inf\{t\geq 0~|~ X_t=\Delta\}=D_{\{\Delta\}}$. Hence the life time $\zeta$ is an $\{\mathcal{M}_t\}$-stopping time. In fact $\zeta$ is an $(\mathcal{F}^0_{t^+})_{t\geq0}$-stopping time. This can be seen by the fact that $\{\zeta <t\}=\bigcup_{s\in\mathds{Q}\cap(0,t)}\{X_s=\Delta\}\in \mathcal{F}^0_t$ (here $\mathds{Q}$ stands for rational numbers).
\end{remark}

For an $\{\mathcal{M}^{\mu}_t\}$-stopping time $\sigma$, we denote by
\begin{equation}\label{ali-Mmut strong}
\mathcal{M}^{\mu}_{\sigma}:=\{\Lambda\in\mathcal{M}^{\mu}_{\infty}|
~\Lambda\cap\{\sigma\leq t\}\in\mathcal{M}^{\mu}_t, ~\forall~ t\geq0\},
\end{equation}
and for an $\{\mathcal{M}_t\}$-stopping time $\sigma$, we denote by
\begin{equation}\label{ali-Mmut strong all}
\mathcal{M}_{\sigma}:=\{\Lambda\in\mathcal{M}_{\infty}|
~\Lambda\cap\{\sigma\leq t\}\in\mathcal{M}_t, ~\forall~ t\geq0\}.
\end{equation}
We make the convention that $X_\infty = \Delta.$ The remark below can be checked by standard arguments (cf. e.g. \cite[3.12]{HWY92} and \cite[Lemma A.1.13 (ii)]{Chen_Fukushima} ).
\begin{remark}\label{remark-X_t measureble stopping time 2}
(i) If $\sigma$ is an $\{\mathcal{M}^{\mu}_t\}$-stopping time, then $X_{\sigma}\in\mathcal{M^{\mu}_{\sigma}}/\mathcal{B}(E_{\Delta}).$

(ii) If $\sigma$ is an $\{\mathcal{M}_t\}$-stopping time, then $X_{\sigma}\in\mathcal{M}_{\sigma}/\mathcal{B}(E_{\Delta})$.
\end{remark}
The second assertion of the above remark can be strengthened as the proposition below.

\begin{proposition}\label{pro-unversal}
If $\sigma$ is an $\{\mathcal{M}_t\}$-stopping time, then $X_{\sigma}\in\mathcal{M}_{\sigma}/\mathcal{B}(E_{\Delta})^*$.
\end{proposition}
\begin{proof}
By Remark \ref{remark-X_t measureble stopping time 2}, for  $C\in\mathcal{B}(E_{\Delta})$ and $t\geq 0$, we have $\{X_{\sigma}\in C\}\cap\{\sigma\leq t\}\in\mathcal{M}_t.$ Hence,  for any $\mu\in\mathcal{P}(E_{\Delta})$, it follows from  Lemma \ref{thm-relation M and F} that $\{X_{\sigma}\in C\}\cap\{\sigma\leq t\}\cap\{t<\zeta\}\in\mathcal{F}^{h,\mu}_t$.
We write $\nu(C):= P^{h}_{\mu}(X_{\sigma}\in C;\sigma\leq t,t<\zeta)$ for $C\in\mathcal{B}(E_{\Delta})$, then $\nu$ is a finite measure on $\mathcal{B}(E_{\Delta})$.

For any $A\in\mathcal{B}(E_{\Delta})^*$, there exist $A',B\in\mathcal{B}(E_{\Delta})$ such that $A\triangle A'\subset B$ and $\nu(B)=0$. Then,
\begin{align*}
(\{X_{\sigma}\in A\}&\triangle\{X_{\sigma}\in A'\})\cap\{\sigma\leq t\}\\
&=\{X_{\sigma}\in A \triangle A'\}
\cap\{\sigma\leq t\}
\subset\{X_{\sigma}\in B\}\cap\{\sigma\leq t\},
\end{align*}
and
$$P^{h}_{\mu}(\{X_{\sigma}\in B\}\cap \{\sigma\leq t\};t<\zeta)=\nu(B)=0.$$
By Lemma \ref{lem-Mmut} we see that $\{X_{\sigma}\in A\}\cap\{\sigma\leq t\}\in\mathcal{M}^{\mu}_t$. Since $\mu\in\mathcal{P}(E_{\Delta})$ is arbitrary, hence $\{X_{\sigma}\in A\}\cap\{\sigma\leq t\}\in\mathcal{M}_t$, which means $\{X_{\sigma}\in A\}\in\mathcal{M}_{\sigma}$.
\end{proof}

For an $\{\mathcal{M}^{\mu}_t\}$-stopping time $\sigma,$ we set
\begin{equation}\label{sigmazeta}
\sigma_{\zeta}:=\sigma I_{\{\sigma<\zeta\}}+ \infty I_{\{\sigma \geq \zeta\}}.
\end{equation}
Recall $(\mathcal{F}^{h,\mu}_{t})_{t\geq 0}$ is the $P^h_{\mu}$-augmentation  of the filtration $(\mathcal{F}^{0}_{t})_{t\geq 0}$ in $\mathcal{F}^{h}_{\infty}$ (cf. (\ref{ali-Fhtmu}) ).
\begin{lemma}\label{thm-sigmazeta}
(i) If $\sigma$ is  an $\{\mathcal{M}^{\mu}_t\}$-stopping time, then for any $h\in \mathcal{H},$   $\sigma_{\zeta}$ is  an $(\mathcal{F}^{h,\mu}_{t})_{t\geq 0}$-stopping time, and $\Lambda\cap\{\sigma<\zeta\}\in\mathcal{F}^{h,\mu}_{\sigma_{\zeta}}$ for $\Lambda\in\mathcal{M}^{\mu}_{\sigma}$.

(ii) If $\sigma$ is an $\{\mathcal{M}_t\}$-stopping time and $\Lambda\in\mathcal{M}_{\sigma}$, then for any $h\in \mathcal{H},$ $\sigma_{\zeta}$ is an $(\mathcal{F}^{h}_{t})_{t\geq 0}$-stopping time and $\Lambda\cap\{\sigma<\zeta\}\in\mathcal{F}^{h}_{\sigma_{\zeta}}$ .
\end{lemma}
\begin{proof} Assertion (ii) follows directly from Assertion (i). Below we prove only Assertion (i). Let $\sigma$ be an $\{\mathcal{M}^{\mu}_t\}$-stopping time and $\Lambda\in\mathcal{M}^{\mu}_{\sigma}.$ If $\sigma$ takes only discrete values, for  $a\in\mathds{R}_+$,  by Lemma \ref{thm-relation M and F} (i), we have for any $h\in \mathcal{H},$
\begin{align*}
\Lambda\cap\{\sigma_{\zeta}=a\}&=
[\Lambda\cap\{\sigma<\zeta\}]\cap\{\sigma=a\}\\
&=[\Lambda\cap\{\sigma=a\}]
\cap\{a<\zeta\}\in\mathcal{F}^{h,\mu}_a,
\end{align*}
 which implies that $\sigma_{\zeta}$ is an $(\mathcal{F}^{h,\mu}_{t})_{t\geq 0}$-stopping time and $\Lambda\cap\{\sigma<\zeta\}\in\mathcal{F}^{h,\mu}_{\sigma_{\zeta}}.$
In general case, we set
\begin{equation}\label{sigman}
\sigma_n =
\begin{cases} \frac{k}{2^n}, &~\mbox{for all} ~~ \frac{k-1}{2^n}\leq\sigma<\frac{k}{2^n},\\
+\infty , &~~ \sigma=+\infty,
\end{cases}
\end{equation}
where $k=1,2,\dots$, $n=1,2,\dots$.
Then $\sigma_n(\omega)$ decreases to $\sigma(\omega)$ and  $\sigma_{n,\zeta}(\omega)$ decreases to $\sigma_{\zeta}(\omega)$ as $n\to+\infty$. Since $\{\mathcal{M}^{\mu}_t\}$ is right continuous, by a routine argument we can show that $\sigma_{\zeta}$ is an $(\mathcal{F}^{h,\mu}_{t})_{t\geq 0}$-stopping time and $\Lambda\cap\{\sigma<\zeta\}\in\mathcal{F}^{h,\mu}_{\sigma_{\zeta}}$ .
\end{proof}

\begin{lemma}\label{lemma-strong pseudo Markov01}
Let $\mu\in\mathcal{P}(E_{\Delta})$ and $\sigma$ be an $\{\mathcal{M}^{\mu}_t\}$-stopping time. If $\Lambda\in\mathcal{M}^{\mu}_{\sigma}$ with $P^h_{\mu}(\Lambda;\sigma<\zeta)=0$, then for any $\Lambda'\subset\Lambda$, we have $\Lambda'\cap\{\sigma<\zeta\}\in\mathcal{M}^{\mu}_{\sigma}$.
\end{lemma}
\begin{proof}
 Suppose first that $\sigma$ is a discrete $\{\mathcal{M}^{\mu}_t\}$-stopping time taking values in $\{a_1,a_2,\dots,a_n,\dots,+\infty\}$.
Then,
\begin{align*}
0
=P^h_{\mu}(\Lambda;\sigma<\zeta)
=\sum^{+\infty}_{i=1}&P^h_{\mu}(\Lambda;\sigma=a_i,\sigma<\zeta)\\
&=\sum^{+\infty}_{i=1}P^h_{\mu}(\Lambda\cap\{\sigma=a_i\};a_i<\zeta).
\end{align*}
by the completeness of $\mathcal{M}^{\mu}_{a_i}$ with respect to $P^h_{\mu}(\,\cdot\,;a_i<\zeta)$, for any $\Lambda'\subset\Lambda$,
we have
\begin{center} $\Lambda'\cap\{\sigma<\zeta\}\cap\{\sigma=a_i\}=\Lambda'\cap
\{\sigma=a_i<\zeta\}\in\mathcal{M}^{\mu}_{a_i}, ~~\forall ~~i\geq 1.$
\end{center}
Hence $\Lambda'\cap\{\sigma<\zeta\}\in\mathcal{M}^{\mu}_{\sigma}.$

For general $\{\mathcal{M}^{\mu}_t\}$-stopping time $\sigma$, there exists a sequence of discrete $\{\mathcal{M}^{\mu}_t\}$-stopping time $(\sigma_n)_{n\geq1}$ such that
$\sigma_n\downarrow \sigma$ as $n\to+\infty$ (cf (\ref{sigman})). Then $\Lambda\in\mathcal{M}^{\mu}_{\sigma}$ and $P^h_{\mu}(\Lambda;\sigma<\zeta)=0$ implies $\Lambda\in\mathcal{M}^{\mu}_{\sigma_n}$ and $P^h_{\mu}(\Lambda;\sigma_n<\zeta)=0$ for all $n\geq 1.$
For any $\Lambda'\subset\Lambda,$ we have
\begin{center}
$\Lambda'\cap\{\sigma<\zeta\}\cap\{\sigma<t\}=
\bigcup\limits^{\infty}_{n=1}[\Lambda'\cap\{\sigma_n<\zeta\}\cap\{\sigma_n<t\}]
\in\mathcal{M}^{\mu}_{t}$.
\end{center}
Therefore, by the right continuity of $\{\mathcal{M}^{\mu}_{t}\}$ we get  $\Lambda'\cap\{\sigma<\zeta\}\in\mathcal{M}^{\mu}_{\sigma}.$
\end{proof}

\begin{lemma}\label{lemma-strong pseudo Markov02}
Let $\sigma$ be an $\{\mathcal{M}_t\}$-stopping time and $t\geq0$, then $\theta^{-1}_{\sigma}\mathcal{M}_t\cap\{t+\sigma <\zeta\}\subset\mathcal{M}_{t+\sigma}$.
\end{lemma}
\begin{proof}
Firstly, let $\Lambda=\bigcap\limits^{n}_{i=1}\{X_{t_i}\in B_i\}$ for some $B_i\in\mathcal{B}(E_{\Delta}), i=1,2,\dots,n,$ and $0\leq t_1<t_2<\dots<t_n\leq t$.
 Then by Remark \ref{remark-X_t measureble stopping time 2} we get
\begin{center}
$\theta^{-1}_{\sigma}\Lambda=\bigcap\limits^{n}_{i=1}\{X_{t_i+\sigma}\in B_i\}\in\mathcal{M}_{t+\sigma}.$
\end{center}
By monotone class argument we conclude that  $\theta^{-1}_{\sigma}\mathcal{F}^0_t\subset\mathcal{M}_{t+\sigma},$ and consequently $\theta^{-1}_{\sigma}\mathcal{F}^0_{t^+}\subset\mathcal{M}_{t+\sigma}$ because $\{\mathcal{M}_{t}\}$ is right continuous.

Secondly, for any $\mu\in\mathcal{P}(E_{\Delta})$, we can define a
finite measure $\nu$ on $\mathcal{B}(E_{\Delta})$ by setting $\nu(A):=P^h_{\mu}(X_{\sigma_{\zeta}}\in A;\sigma_{\zeta}<\zeta)$ for $A\in\mathcal{B}(E_{\Delta})$. Then for any $\Lambda\in\mathcal{M}_t\subset\mathcal{M}^{\nu}_t$, there exist $\Lambda'\in\mathcal{F}^0_t$ and $\Gamma\in\mathcal{F}^0_{t^+}$ such that
$\Lambda\Delta\Lambda'\subset\Gamma$ and $P^h_{\nu}(\Gamma;t<\zeta)=0$. By the conclusion which we have just proved above, it holds that
$\theta^{-1}_{\sigma}\Gamma\in\mathcal{M}_{t+\sigma}$.
We have
\begin{align*}
P^h_{\mu}(\theta^{-1}_{\sigma}\Gamma;&t+\sigma<\zeta)
=P^h_{\mu}(\theta^{-1}_{\sigma}\Gamma;\sigma<\zeta,t<\zeta\comp\theta_{\sigma})
\\
&=P^h_{\mu}(\theta^{-1}_{\sigma_{\zeta}}\Gamma;\sigma_{\zeta}<\zeta,
t<\zeta\comp\theta_{\sigma_{\zeta}})\\
&=E^h_{\mu}[P^h_{X_{\sigma_{\zeta}}}(\Gamma;t<\zeta);\sigma_{\zeta}<\zeta]
=P^h_{\nu}(\Gamma;t<\zeta)=0.
\end{align*}
Hence, because  $(\theta^{-1}_{\sigma}\Lambda)\Delta(\theta^{-1}_{\sigma}\Lambda')
 =\theta^{-1}_{\sigma}(\Lambda\Delta\Lambda')\subset\theta^{-1}_{\sigma}\Gamma,$ by Lemma \ref{lemma-strong pseudo Markov01} we get $[(\theta^{-1}_{\sigma}\Lambda)\setminus(\theta^{-1}_{\sigma}\Lambda')]
\cap\{t+\sigma< \zeta\}\in\mathcal{M}^{\mu}_{t+\sigma}$ and
 $[(\theta^{-1}_{\sigma}\Lambda')\setminus(\theta^{-1}_{\sigma}\Lambda)]
\cap\{t+\sigma< \zeta\}\in\mathcal{M}^{\mu}_{t+\sigma}.$  From these two facts together with the fact that $\theta^{-1}_{\sigma}\Lambda'\in\mathcal{M}_{t+\sigma}$ we get
$\theta^{-1}_{\sigma}\Lambda\cap\{t+\sigma<\zeta\}\in\mathcal
{M}^{\mu}_{t+\sigma}.$  Consequently, $\theta^{-1}_{\sigma}\Lambda\cap\{t+\sigma<\zeta\}\in\mathcal{M}_{t+\sigma}$
because $\mu\in\mathcal{P}(E_{\Delta})$ is arbitrary.
\end{proof}

Below we denote by $\mathcal{T}$ the collection of all the $\{\mathcal{M}_t\}$-stopping times.
\begin{definition}\label{def-exdistribution}
 For $\sigma \in \mathcal{T}, x\in E,$  we define a measure $\overline{Q}_{x,\sigma}$ on $\mathcal{M}_{\sigma}$ by setting
\begin{equation}\label{eq-exdistribution}
\overline{Q}_{x,\sigma}(\Lambda):= \overline{Q}_{x,\sigma}(\Lambda; \sigma<\zeta ):=h(x)E^{h}_{x}(\frac{e^{\alpha \sigma}I_{\Lambda}}{h(X_{\sigma})}; \sigma<\zeta), ~ \forall ~ \Lambda\in\mathcal{M}_{\sigma},
\end{equation}
where $I_{\Lambda}$ is the indicator function of $\Lambda$, $E^{h}_{x}$ is the expectation related to $P^h_x$. We call ~$\overline{Q}_{x,{\sigma}}$ a {\bf \it $\sigma$-finite distribution up to time $\sigma$} (in short, distribution up to $\sigma$), and call ~$(\overline{Q}_{x,\sigma})_{\sigma \in \mathcal{T}}$ an {\bf \it expanded  $\sigma$-finite distribution flow (in short, expanded distribution flow) associated with $(\mathcal{E}, D(\mathcal{E}))$}.
\end{definition}
\begin{proposition}\label{pro-exdistribution}
(i) Let $\sigma \in \mathcal{T}$. Then $\overline{Q}_{x,\sigma}(\cdot)$ is a well defined $\sigma$-finite measure on $\mathcal{M}_{\sigma}$ for fixed $ x\in E,$
and $\overline{Q}_{x,\sigma}(\Lambda; \sigma<\zeta )$ being a function of $x$ is $\mathcal{B}(E_{\Delta})^*$ measurable for fixed $\Lambda\in\mathcal{M}_{\sigma}$.

(ii) The definition of $\overline{Q}_{x,\sigma}$ is independent of the choice of $h\in \mathcal{H}.$  More precisely, if $\overline{Q}_{x,\sigma}'$ is defined by (\ref{eq-exdistribution}) with $h$  replaced by another $h' \in \mathcal{H}.$ Then $\overline{Q}_{x,\sigma}$ and $\overline{Q}'_{x,\sigma}$ are equivalent to each other. That is, there exists a Borel set $S\subset E$ such that $\overline{Q}_{x,\sigma}(\cdot)$ and $\overline{Q}_{x,\sigma}'(\cdot)$ are identical for all $x\in S$ and $E\setminus S$ is $\mathcal{E}$-exceptional (cf. Definition \ref{equivalent}).
\end{proposition}
\begin{proof}
(i) By Lemma \ref{thm-sigmazeta}, $\Lambda\cap\{\sigma<\zeta\}\in\mathcal{F}^{h}_{\sigma_{\zeta}}$ for $\Lambda\in\mathcal{M}_{\sigma}$, therefore $\overline{Q}_{x,\sigma}$ is well defined. Let $\Lambda_k=\{\sigma<k, X_{\sigma}\in F_k, X_0\in F_k\}$, then $\overline{Q}_{x,\sigma}(\Lambda_k)< \infty,$ hence $\overline{Q}_{x,\sigma}$ is a $\sigma$-finite measure. The last assertion follows from Lemma \ref{thm-sigmazeta} and the standard theory of Markov processes (cf. e.g. \cite[Chapter I,(5.8)]{BG_markov_and_potential}).

(ii)  To check that $\overline{Q}_{x,\sigma}$ is independent of the choice of $h \in \mathcal{H},$ let $\overline{Q}_{x,\sigma}'$ be defined by (\ref{eq-exdistribution}) with $h$  replaced by another $h' \in \mathcal{H}.$ When $\sigma \in \mathcal{T}$ is a discrete stopping time, making use of  Theorem \ref{thm-independent} we can check that  $\overline{Q}'_{x,\sigma}(\cdot)$ and $\overline{Q}_{x,\sigma}(\cdot)$ are equivalent to each other.
For general $\sigma \in \mathcal{T},$ we define $\sigma_n$ in the same manner as (\ref{sigman}) above. Let $\Lambda_k=\{\sigma<k, X_{\sigma}\in F_k, X_0\in F_k\}$, then for any $\Lambda\in\mathcal{M}_{\sigma},$ we have for $x\in S,$
$$\overline{Q}_{x,\sigma}(\Lambda \cap \Lambda_k) = \lim_{n\rightarrow \infty}\overline{Q}_{x,\sigma_n}(\Lambda \cap \Lambda_k)= \lim_{n\rightarrow \infty}\overline{Q}'_{x,\sigma_n}(\Lambda \cap \Lambda_k)= \overline{Q}'_{x,\sigma}(\Lambda \cap \Lambda_k),$$
where $S$ is specified by Lemma \ref{invariantSet}.  Letting $k$ tends to infinity, we get $\overline{Q}_{x,\sigma}(\Lambda)= \overline{Q}'_{x,\sigma}(\Lambda)$  for $x\in S.$  Consequently $\overline{Q}'_{x,\sigma}(\cdot)$ and $\overline{Q}_{x,\sigma}(\cdot)$ are equivalent to each other.
\end{proof}

We are now in a position to state the strong Markov property of the distribution flows. The readers may compare the theorem below with Theorem \ref{thm-Pseudo-Markov} in Section \ref{section-distributionFlow}.

\begin{theorem}\label{thm-strong pseudo Markov}
 Let $\sigma \in \mathcal{T}$ and $ \tau \in \mathcal{T}.$ We set  $$\gamma^*=\sigma + \tau\comp\theta_{\sigma}~~\mbox{ and }~~ \gamma=\gamma^*_{\zeta}:= \gamma^*I_{\{\gamma^*<\zeta\}}+ \infty I_{\{\gamma^*\geq \zeta\}}.$$

 (i) It holds that  $\gamma \in\mathcal{T}$ and $Y\comp\theta_{\sigma} \in\mathcal{M}_{\gamma}$ for  $Y \in p\mathcal{M}_{\tau}$.

 (ii) For $\Gamma \in p\mathcal{M}_{\sigma}$ and $Y \in p\mathcal{M}_{\tau}$,  we have
\begin{equation}\label{equ-strong_pseudo_Markov01}
\overline{Q}_{x,\gamma}[\Gamma~(Y\comp\theta_{\sigma});\gamma<\zeta]=
\overline{Q}_{x,\sigma}[\Gamma~ \overline{Q}_{X_\sigma,\tau}[Y;
\tau <\zeta]; \sigma<\zeta].
\end{equation}
In particular, if $\tau=u$ is a constant, then for $\Gamma \in p\mathcal{M}_{\sigma}$ and $f\in \mathcal{B}(E)^+$,
we have
\begin{equation}
\overline{Q}_{x,\sigma +u}[\Gamma~ f(X_{\sigma +u});\sigma +u <\zeta]=\overline{Q}_{x,\sigma}[\Gamma ~ (Q_{u}f)(X_\sigma);\sigma<\zeta].
\end{equation}
\end{theorem}
\begin{proof}
(i) For $\Lambda\in\mathcal{M}_{\tau}$ and  $t>0$, we have
$$\theta^{-1}_{\sigma}\Lambda\cap\{\gamma<t\}=
\bigcup\limits_{p\in\mathds{Q}}
[\theta^{-1}_{\sigma}\Lambda\cap\{\tau\comp\theta_{\sigma}<p\}
\cap\{p+\sigma<t\}\cap\{p+\sigma<\zeta\}],$$
where $\mathds{Q}$ is the collection of all the rational numbers.
Note that $\Lambda\cap\{\tau<p\}\in\mathcal{M}_p$, by Lemma \ref{lemma-strong pseudo Markov02} we get
\begin{align*}
\theta^{-1}_{\sigma}\Lambda&\cap\{\tau\comp\theta_{\sigma}<p\}\cap
\{p+\sigma<\zeta\}\\
&=\theta^{-1}_{\sigma}[\Lambda\cap\{\tau<p\}]\cap
\{p+\sigma<\zeta\}\in\mathcal{M}_{p+\sigma}.
\end{align*}
Thus,
\begin{align}\label{gamma}
\theta^{-1}_{\sigma}\Lambda&\cap\{\gamma<t\}\\
&=\bigcup\limits_{p\in\mathds{Q}}
[\theta^{-1}_{\sigma}\Lambda\cap\{\tau\comp\theta_{\sigma}<p\}
\cap\{\sigma+p<t\}\cap\{p+\sigma<\zeta\}]\in\mathcal{M}_t.\nonumber
\end{align}
Letting $\Lambda=\Omega$ in (\ref{gamma}) we get $\{\gamma<t\}\in\mathcal{M}_t$, which means $\gamma$ is an $\{\mathcal{M}_t\}$-stopping time because $\{\mathcal{M}_t\}$ is right continuous. By (\ref{gamma}) we conclude also that $\theta^{-1}_{\sigma}\Lambda\cap\{\gamma<\zeta\} \in\mathcal{M}_{\gamma}$ for any $\Lambda\in\mathcal{M}_{\tau}$, and consequently $Y\comp\theta_{\sigma}\cdot I_{(\gamma<\zeta)}\in\mathcal{M}_{\gamma}$ for any $Y \in p\mathcal{M}_{\tau}$.

(ii) Let $\Gamma \in p\mathcal{M}_{\sigma}$ and $Y \in p\mathcal{M}_{\tau}.$ By Lemma \ref{thm-sigmazeta} and the strong Markov property of $((X_t)_{t\geq0},(\mathcal{F}^h_t)_{t\geq0}, P^h_x)$ (cf. \cite[Chapter I.(8.6)]{BG_markov_and_potential}),  we get
\begin{align*}
\overline{Q}_{x,\gamma}[\Gamma~ &(Y\comp\theta_{\sigma});\gamma<\zeta]
=h(x)E^h_x[\frac{e^{\alpha(\sigma+\tau\comp\theta_{\sigma})}
Y\comp\theta_{\sigma}\Gamma}{h(X_{\sigma+\tau\comp\theta_{\sigma}})};
\sigma+\tau\comp\theta_{\sigma}<\zeta]\\
&=h(x)E^h_x[\frac{e^{\alpha(\sigma_{\zeta}+\tau_{\zeta}
\comp\theta_{\sigma_{\zeta}})}Y\comp\theta_{\sigma_{\zeta}}\Gamma}
{h(X_{\sigma_{\zeta}+\tau_{\zeta}\comp\theta_{\sigma_{\zeta}}})};
\sigma_{\zeta}<\zeta,\tau_{\zeta}\comp\theta_{\sigma_{\zeta}}<\zeta
\comp\theta_{\sigma_{\zeta}}]\\
&=h(x)E^h_x[e^{\alpha\sigma_{\zeta}}\Gamma I_{(\sigma_{\zeta}<\zeta)}
E^h_x(\frac{e^{\alpha\tau_{\zeta}}YI_{(\tau_{\zeta}<\zeta)}}
{h(X_{\tau_{\zeta}})}\comp\theta_{\sigma_{\zeta}}|\mathcal{F}^h
_{\sigma_{\zeta}})]\\
&=h(x)E^h_x[e^{\alpha\sigma_{\zeta}}\Gamma I_{(\sigma_{\zeta}<\zeta)}
E^h_{X_{\sigma_{\zeta}}}(\frac{e^{\alpha\tau_{\zeta}}Y}
{h(X_{\tau_{\zeta}})};\tau_{\zeta}<\zeta)]\\
&=h(x)E^h_x[\frac{e^{\alpha\sigma}\Gamma I_{(\sigma<\zeta)}}
{h(X_{\sigma})}h(X_{\sigma})E^h_{X_{\sigma}}(\frac{e^{\alpha\tau}Y}
{h(X_{\tau})};\tau<\zeta)]\\
&=\overline{Q}_{x,\sigma}[\Gamma~ \overline{Q}_{X_\sigma,\tau}[Y;
\tau <\zeta]; \sigma<\zeta].
\end{align*}
The above last equality made use of the last assertion of Proposition \ref{pro-exdistribution} (i) and Proposition \ref{pro-unversal}.
\end{proof}

\section{$\mathcal{O}$-measurable positive continuous additive functionals}\label{Sec-PCAF}

\subsection{Preliminaries and Definition}\label{Sec-PCAF1}
  Let $(\mathcal{E},D(\mathcal{E}))$ be a quasi-regular positivity preserving coercive form on $L^2(E;m).$
Following \cite{FOT11} (see also \cite{Oshima13}), a positive measure $\mu$ on $(E,\mathcal{B}(E))$ will be called smooth w.r.t. $(\mathcal{E},D(\mathcal{E}))$, and be denoted by $\mu\in\mathcal{S}$,
if $\mu(N)=0$ for each $\mathcal{E}$-exceptional set $N\in\mathcal{B}(E)$ and there exists an $\mathcal{E}$-nest $(F_k)_{k\geq 1}$ of compact subsets of $E$ such that
$\mu(F_k)<\infty$ for all $k\in\mathds{N}$.  A positive Radon measure $\mu$ on $(E,\mathcal{B}(E))$ is said to be of finite energy integral  w.r.t. $(\mathcal{E},D(\mathcal{E}))$, denoted by $\mu\in\mathcal{S}_0$, if $\mu\in \mathcal{S}$ and there exists $C>0$ such that
\begin{center}
$\int_E|\tilde{v}(x)|\mu(dx)\leq C\mathcal{E}_1(v,v)^{\frac{1}{2}}$ for all $v\in D(\mathcal{E})$.
\end{center}
Let $(\mathcal{E}^h,D(\mathcal{E}^h))$ be an $h$-transform of $(\mathcal{E},D(\mathcal{E}))$ with some $h\in \mathcal{H}.$ Then one can easily check that $\mu$ is a smooth measure w.r.t. $(\mathcal{E},D(\mathcal{E}))$  if and only if it is a  smooth measure w.r.t. the semi-Dirichlet form $(\mathcal{E}^h, D(\mathcal{E}^h))$. We shall denote by $\mathcal{S}^h_0$ all the measures of finite energy integral w.r.t. $(\mathcal{E}^h, D(\mathcal{E}^h))$.

For $\mu\in S_0$, applying a theorem of G. Stampacchia \cite{Stampacchia64} (cf. \cite[I.Theorem 2.6]{MR92}) we can show that there exists a unique $U_{\alpha}\mu\in D(\mathcal{E})$ and an unique $\hat{U}_{\alpha}\mu\in D(\mathcal{E})$ such that
\begin{center}
$\mathcal{E}_{\alpha}(U_{\alpha}\mu,v)=\int_E\tilde{v}(x)\mu(dx)=\mathcal{E}_{\alpha}(v,\hat{U}_{\alpha}\mu)$ for all $v\in D(\mathcal{E})$.
\end{center}
We call $U_{\alpha}\mu$ (rep. $\hat{U}_{\alpha}\mu$) an $\alpha$-potiential (resp. $\alpha$-copotential) of $\mu\in S_0$ w.r.t. $(\mathcal{E},D(\mathcal{E}))$. For notational convenience, we shall denote by $U^h_{\alpha}\mu$ (resp. $\hat{U}^h_{\alpha}\mu$) the $\alpha$-potiential (resp. $\alpha$-copotential) of $\mu\in S_0^{h}$ w.r.t. $(\mathcal{E}^h,D(\mathcal{E}^h))$.
The following lemma can be checked directly and we omit their proofs.
\begin{lemma}\label{lemma-S_0 and S^h_0}
(i) If $\mu\in\mathcal{S}_0$, then $h\cdot \mu \in\mathcal{S}^h_0$. If $\nu\in\mathcal{S}^h_0$, then $(h^{-1})\cdot \nu\in\mathcal{S}_0$.

(ii) \label{lemma-potential measure between h trans}
For any $\mu\in\mathcal{S}_0$, $\beta\geq \alpha$, $U_{\beta}\mu=hU^h_{\beta-\alpha}(h\cdot \mu)$, $m$-a.e..

(iii)\label{SS0compact}
The following  two assertions are equivalent to each other.

(a) $\mu\in\mathcal{S}.$

(b) There exists an $\mathcal{E}$-nest $(K_n)_{n\geq1}$ consisting of compact sets such that $I_{K_n}\cdot\mu\in\mathcal{S}_{0}$ for each $n\in\mathds{N}$.
\end{lemma}

In what follows we use $\mathds{R}_+$ to denote $[0, \infty)$ and use $\overline{\mathds{R}}_+$ to denote $[0, \infty].$
Let $\mathcal{O}$ be the optional $\sigma$-field related to the filtration $(\mathcal{M}_t)_{t\geq 0}$. That is, $\mathcal{O}$ is the $\sigma$-field on $[[0, \infty[[:=\mathds{R}_+\times \Omega$ generated by all the $\{\mathcal{M}_t\}$-adapted cadlag processes. It is known that (cf. \cite[Theorem 3.17]{HWY92})
\begin{equation}\label{defoptionalfieldall}
\mathcal{O}:=\sigma\{[[T, \infty[[~|~T\in \mathcal{T} \}.
\end{equation}
Here and henceforth, $[[T, \infty[[:= \{(t,\omega)~|~ T(\omega)\leq t < \infty\},$
$\mathcal{T}$ is the collection of all the $\{\mathcal{M}_t\}$-stopping times.
An $\mathcal{O}$-measurable process is called an optional process.

For $\mu \in \mathcal{P}(E_\Delta),$ we define a $\sigma$-finite measure $\mathbb{Q}_{\mu}$ on $\mathcal{O}$ as follows.
\begin{equation}\label{def-optional measure-mu}
\mathbb{Q}_{\mu}(H):=\int^{+\infty}_0\overline{Q}_{\mu,t}(I_H(t,\,\cdot\,))dt, ~~\forall~H\in\mathcal{O}.
\end{equation}
In particular, for  $\mu = \delta_x$ we write
\begin{equation}\label{def-optional measure-x}
\mathbb{Q}_{x}(H):=\int^{+\infty}_0\overline{Q}_{x,t}(I_H(t,\,\cdot\,))dt, ~~\forall~H\in\mathcal{O}.
\end{equation}

\begin{definition}\label{def-PCAF2}
 (i) An $\overline{\mathds{R}}_+$-valued optional process $A:=(A_t)_{t\geq 0}$ is called an $\mathcal{O}$-measurable positive continuous additive functional ($\mathcal{O}$-PCAF in abbreviation) if there exists a  defining set $\Gamma\in\mathcal{O}$ such that:

 (a) $I_{\Gamma}(t,\omega)$ is deceasing and right continuous in $t$ for fixed $\omega,$ and $I_{\Gamma}(t+s,\omega)=1$ implies $I_{\Gamma}(s,\theta_t\omega)=1;$

 (b) $\mathbb{Q}_{\nu}(\Gamma^c)=0$ for all $\nu \in \mathcal{S}_0,$ where $\Gamma^c:= [[0, \infty)) \setminus \Gamma;$

 (c)  Let $\tau_{\Gamma}(\omega):=\inf\{t\geq0~|~ (t, \omega) \notin \Gamma\}$ and $\Lambda_{\Gamma}:=\{\omega~|~\tau_{\Gamma}(\omega)\geq \zeta(\omega)\}, $ then $\Lambda_{\Gamma}=\{\omega~|~\tau_{\Gamma}(\omega)=\infty\}.$\\
 Furthermore, the restriction of $A$ on $\Gamma$, or equivalently, the restriction of $A$ on $\{\tau_{\Gamma}>0\}$, possesses the following properties:

 (d) $A_t$ is continuous for $0\leq t < \tau_{\Gamma}, ~A_0=0,~ A_t<\infty$ for $t<\tau_{\Gamma}\wedge \zeta,$ and $A_t=(A_{\zeta})_{-}$ for $t\geq \zeta;$

 (e) $A_{t+s}(\omega)=A_t(\omega)+A_s(\theta_t\omega)$ for $t+s < \tau_{\Gamma}(\omega).$

(ii) Two $\mathcal{O}$-PCAF $A$ and $A'$ are said to be $\mathcal{O}$-equivalent if they share a common defining set $\Gamma$ and their restriction on $\Gamma$ are identical.
\end{definition}
\begin{proposition}\label{classic}
Let $A$ be an $\mathcal{O}$-PCAF and $\Lambda_{\Gamma}$ be specified as in (c) of the above Definition \ref{def-PCAF2}. Then for any $h\in \mathcal{H}$, the restriction of $A$ on  $\Lambda_{\Gamma}$ is a PCAF of $\mathbf{M}^h$ in the classical sense defined in  \cite[Section 5.1]{FOT11} or  \cite[Section 4.1]{Oshima13}, with defining set $\Lambda_{\Gamma}$ and some exceptional set $N$.
\end{proposition}
\begin{proof}
Let $h \in \mathcal{H},$ by Lemma \ref{thm-sigmazeta}  we have $(\Lambda_{\Gamma})^c= \{\tau_{\Gamma} < \zeta\}\in \mathcal{F}^h_{(\tau_{\Gamma})_{\zeta}}\subset \mathcal{F}^h_{\infty}.$ By (b) there exists an $\mathcal{E}$-exceptional set $N$ such that for  $x\in E\setminus N,$
\begin{equation}
\mathbb{Q}_{x}(\Gamma^c):=\int^{\infty}_0h(x)E^h_x[\frac{e^{\alpha t}I_{\Gamma^c}(t,\cdot)}{h(X_t)};t<\zeta]dt=0,
\end{equation}
which implies $E^h_x[I_{\Gamma^c}(t,\cdot);t<\zeta]=E^h_x[I_{\{ \tau_{\Gamma}\leq t\}};t<\zeta]=0$ for almost all $t\in \mathds{R}^+$ and hence $P^h_x((\Lambda_{\Gamma})^c)=E^h_x[I_{\{ \tau_{\Gamma}<\zeta\}}]=0$.
Since $\Lambda_{\Gamma}=\{\omega~|~\tau_{\Gamma}(\omega)=\infty\},$ hence by (a) we have
$\theta_t\Lambda_{\Gamma}\subset \Lambda_{\Gamma}$ for $t\geq 0.$ Therefore, by (d) and (e) $A$ is a PCAF of $\mathbf{M}^h$ with defining set $\Lambda_{\Gamma}$ and exceptional set $N.$
\end{proof}
\begin{remark}
 Conversely, Let $A^{h}$ be a PCAF of $\mathbf{M}^h$ in the classical sense, then applying Theorem \ref{thm-revuz cor} below, we can construct an $\mathcal{O}$-PCAF
$A$ such that the restriction of $A$ on $\Lambda_{\Gamma}$ as a classical PCAF
is equivalent to $A^{h}$ in the classical sense.
\end{remark}

\subsection{Revuz correspondence}\label{Sec-PCAF2}

In this subsection we fix an $h\in \mathcal{H}.$  Suppose that $A^h$ is a PCAF of $\mathbf{M}^h$ and $\mu^h$ is a smooth measure w.r.t. $(\mathcal{E}^h,D(\mathcal{E}^h)).$  Then by the theory of Dirichlet forms,  $A^h$ and $\mu^h$ are said to be Revuz corresponding to each other, and  $\mu^h$ is called the Revuz measure of $A^h,$ if for any  $\gamma$-coexcessive $(\gamma > 0)$ function $g\in D(\mathcal{E}^h)$ and any bounded $f\in p\mathcal{B}(E),$ it holds that
\begin{equation}\label{eq: classic revuz cor}
\lim\limits_{\beta\to+\infty}\beta(g,~E^h_{\cdot}[\int^{+\infty}_0e^{-(\beta + \gamma) t}
f(X_t)I_{(t<\zeta)}dA^h_t])_{h^2\cdot m}=\int_Efg\mu^h(dx).
\end{equation}
We refer to \cite[Section 5.1]{FOT11} and \cite[Section 4.1]{Oshima13} for the detail discussion of Revuz correspondence. The condition stated above is slightly different but equivalent to the condition stated in \cite[Theorem 4.1.4]{Oshima13} (cf. \cite[Theorem 5.1.3]{FOT11} ).

In this subsection we shall prove the following theorem.
\begin{theorem}\label{thm-revuz cor}
(i) For any $\mathcal{O}$-PCAF $A$, there exists a smooth measure $\mu=\mu_A ,$  such that for any  $\gamma$-coexcessive ($\gamma>\alpha$) function $g\in D(\mathcal{E})$  and any bounded function $f\in p\mathcal{B}(E),$ it holds that
\begin{equation}\label{eq: revuz cor}
\lim\limits_{\beta\to+\infty}\beta(g,U^{\beta+\gamma}_Af)_m=
\int_Efg\mu(dx),
\end{equation}
here and henceforth,
\begin{equation}\label{UbetaA}
U^{\beta}_Af(x):=h(x)E^h_x[\int^{+\infty}_0e^{-(\beta-\alpha)t}
\frac{f}{h}(X_t)I_{(t<\zeta)}dA_t].
\end{equation}
Moreover, if $A$ and $B$ are $\mathcal{O}$-equivalent $\mathcal{O}$-PCAFs, then  $\mu_A$ and $\mu_B$ are identical.

(ii) Conversely, for any $\mu\in\mathcal{S}$, there exists an unique (in $\mathcal{O}$-equivalent sense)  $\mathcal{O}$-PCAF $A$, such that assertion  (\ref{eq: revuz cor}) holds.
\end{theorem}
\begin{remark}
We shall say that $A$ and $\mu$ are Revuz corresponding to each other ( w.r.t. the positivity preserving coercive form $(\mathcal{E},D(\mathcal{E}))$), and $\mu$ is the Revuz measure of $A$, if $A$ and $\mu$ satisfy (\ref{eq: revuz cor}).
\end{remark}

\begin{proof}(Proof of Theorem \ref{thm-revuz cor} (i))

Suppose that $A$ is an $\mathcal{O}$-PCAF  with defining set $\Gamma.$ We define
\begin{equation}\label{AAh}
A^{h}_t(\omega)=\left\{
\begin{array}{ll}
\int^t_0\frac{1}{h(X_s)}dA_s(\omega), & \omega\in \Lambda_{\Gamma}\\
0, & \omega\notin \Lambda_{\Gamma}
\end{array}
\right.
\end{equation}
Then by Proposition \ref{classic}, $A^{h}$ is a PCAF of the $h$-associated process $\mathbf{M}^h$ w.r.t. the semi-Dirichlet space $(\mathcal{E}^h,D(\mathcal{E}^h)).$  By \cite[Section 4.1]{Oshima13}, there exists a Revuz measure $\mu_{A^h}$ associated with $A^h$.  We define $\mu_A:=(h^{-1})\cdot \mu_{A^h}$, then $\mu_A\in\mathcal{S}.$  If $g\in D(\mathcal{E})$ is a $\gamma$-coexcessive function $(\gamma>\alpha)$, then  $\frac{g}{h}\in D(\mathcal{E}^h)$ is a $(\gamma-\alpha)$-coexcessive function in $(\mathcal{E}^h,D(\mathcal{E}^h)).$ For any bounded $f\in \mathcal{B}(E)^+,$ by the Revuz correspondence between
$A^h$ and $\mu_{A^h},$ we have
\begin{align*}
&\,\,\lim\limits_{\beta\to+\infty}\beta(g,U^{\beta+\gamma}_Af)_m\\
&=\lim\limits_{\beta\to+\infty}\beta(g,hE_{\cdot}^h[I_{\Lambda_{\Gamma}}
\int^{+\infty}_0e^{-(\beta+\gamma-\alpha)t}\frac{f}{h}(X_t)
I_{(t<\zeta)}dA_t])_m\\
&=\lim\limits_{\beta\to+\infty}\beta(g,hE_{\cdot}^h[\int^{+\infty}_0
e^{-(\beta+\gamma-\alpha)t}f(X_t)dA^h_t])_m\\
&=\lim\limits_{\beta\to+\infty}\beta(\frac{g}{h},E_{\cdot}^h[\int^{+\infty}_0
e^{-(\beta-\alpha+\gamma)t}f(X_t)dA^h_t])_{h^2\cdot m}\\
&=\int_E\frac{g}{h}f\mu_{A^h}(dx)=\int_Egf\mu_{A}(dx).
\end{align*}
Therefore, (\ref{eq: revuz cor}) is true. Suppose that $B$ is another $\mathcal{O}$-PCAF which is $\mathcal{O}$-equivalent to $A.$ We define $B^h$ in the same manner as (\ref{AAh}). Then one can check that  $B^h$ and $A^h$ are equivalent w.r.t. $(\mathcal{E}^h,D(\mathcal{E}^h))$. Thus, by \cite[Theorem 4.1.4]{Oshima13} we get $\mu_{B^h}=\mu_{A^h}$, consequently
$\mu_B=(h^{-1})\cdot\mu_{B^h}=(h^{-1})\cdot\mu_{A^h}=\mu_A$.
\end{proof}
For proving Theorem \ref{thm-revuz cor} (ii) we prepare  two lemmas first.

\begin{lemma}\label{them-RevuzS0}
Let $\mu\in \mathcal{S}_0.$  Then there exists an unique (in $\mathcal{O}$-equivalent sense) $\mathcal{O}$-PCAF $A$ such that
\begin{align}\label{RevuzS0}
 E_{h\cdot \nu}^{h}[\int_0^{\infty}\frac{e^{-(\beta-\alpha)t} I_{\{t<\zeta\}}}{h(X_t)}dA_t] = \langle U_{\beta}\mu,~\nu \rangle,~~\forall ~ \nu \in \mathcal{S}_0,~ \beta > \alpha,
\end{align}
and consequently, $h(x)E_{x}^{h}[\int_0^{\infty}\frac{e^{-(\beta-\alpha)t} I_{\{t<\zeta\}}}{h(X_t)}dA_t]$ is an $\mathcal{E}$-quasi-continuous version of $U_{\beta}\mu$ for any $\beta>\alpha.$
\end{lemma}

\begin{proof}
Existence.

Let $u$ be an $\mathcal{E}$-quasi-continuous version of $U_{\beta}\mu$ for some temporarily fixed $\beta>\alpha$, then  there exists an $\mathcal{E}$-exceptional set $N$ such that $nR_{n+\beta}u\uparrow u$ on $E\setminus N$. Here, $R_{n+\beta}u:=\int^{+\infty}_0e^{-(n+\beta)t}Q_tu(\cdot)dt$.
Let
\begin{equation}\label{gn}
g_n(x)=\left\{
\begin{array}{ll}
n(u-R_{n+\beta}u)(x), & x\in E\setminus N,\\
0 ,     & x\in N.
\end{array}
\right.
\end{equation}

Then $R_{\beta}g_n\uparrow u$ for any $x\in E\setminus N$.
Let $h\in \mathcal{H},$  we have when $n\to\infty,$
\begin{align}\label{Rgn}
R^h_{\beta-\alpha}\frac{g_n}{h}:=& E_{\cdot}^{h}[\int_0^{\infty}\frac{e^{-(\beta-\alpha)t}g_n(X_t) I_{\{t<\zeta\}}}{h(X_t)}dt]\\
&=\frac{1}{h}R_{\beta}g_n\uparrow \frac{u}{h}=\frac{1}{h}\widetilde{U_{\beta}\mu}
=\widetilde{U^h_{\beta-\alpha}(h\cdot \mu)},\nonumber
\end{align}
where $\widetilde{U^h_{\beta-\alpha}(h\cdot \mu)}$ is an $\mathcal{E}$-quasi-continuous version of the ($\beta-\alpha$)-potential of $h\cdot \mu$ with respect to $(\mathcal{E}^h,D(\mathcal{E}^h)).$
Similar to the proof of
 \cite[Theorem 4.1.10]{Oshima13}, we can choose a subsequence $\{n_l\}$ such that $(R^h_{\beta-\alpha}\frac{f_k}{h})_{k\in\mathds{N}}$ with $f_k:=\frac{1}{k}\sum_{l=1}^{k}g_{n_l}$ converges to $\widetilde{U^h_{\beta-\alpha}(h\mu)}$  strongly in $(\mathcal{E}^h,D(\mathcal{E}^h))$.
Denote by
\begin{equation}\label{Ahtk}
\tilde{A}^h_k(t,\omega):=\int^t_0e^{-(\beta-\alpha)s}\frac{f_k}{h}(X_s)ds.
\end{equation}
Then similar to the argument of  \cite[Theorem 4.1.10]{Oshima13},
we can take a subsequence $k_i$ such that
\begin{equation}\label{subsequence}
 \mathcal{E}^h_{\beta-\alpha}(R^h_{\beta-\alpha}\frac{f_{k_{i+1}}-f_{k_{i}}}{h},
 R^h_{\beta-\alpha}\frac{f_{k_{i+1}}-f_{k_{i}}}{h})<2^{-6i}.
\end{equation}
Let
\begin{align}\label{Lamdah}
\Lambda^h:=\{\omega~|~\tilde{A}^h_{k_i}(t,\omega) \  &\mbox{converges uniformly in}\  t \\
&  \mbox{on each finite interval of} \ [0, \infty)\}, \nonumber
\end{align}
then $P^h_{\nu}((\Lambda^h)^c)=0$ for all  $\nu \in \mathcal{S}_0$  and hence $P^h_{x}((\Lambda^h)^c)=0$ for q.e. $x\in E$ (cf. \cite[(4.1.16)]{Oshima13}).
Let
\begin{equation}\label{def-tildeAh}
\tilde{A}^h(t,\omega)=\left\{
\begin{array}{ll}
\lim\limits_{i\to+\infty}\tilde{A}^h_{k_i}(t,\omega), & \omega\in \Lambda^h,\\
0,      & \omega \notin \Lambda^h .
\end{array}
\right.
\end{equation}
 Denote by $$A^h_t:=\int^t_0e^{(\beta-\alpha)u}d\tilde{A}^h_u,$$
 then $(A^h_t)_{t\geq0}$ is a PCAF in the sense of \cite[Section 5.1]{FOT11} or  \cite[Section 4.1]{Oshima13}, and
its Revuz measure w.r.t $(\mathcal{E}^h,D(\mathcal{E}^h))$ is $h\cdot \mu$.  That is, for any $\beta>\alpha,$ it holds that,
\begin{equation}\label{hrevuz}
U^h_{\beta-\alpha}(h\cdot \mu)=E^h_x[\int^{+\infty}_0e^{-(\beta-\alpha)t}
1(X_t)dA^h_t], \,\,\mbox{q.e.}\,\, x\in E.
\end{equation}
We now define
\begin{align*}\label{Gammat}
\Gamma_t=\bigcap_{m\geq 1}\bigcup_{l\geq m} \{\omega \in \Omega~|&~  \tilde{A}^h_{k_i}(u,\omega) ~
 \mbox{converges uniformly}\\
 &\mbox{ for} \  u\in [0,t+\frac{1}{l}]~ \mbox{when}~ i \rightarrow \infty\},
 \end{align*}
\begin{align*}
\tilde{\Gamma}= \{(t,\omega)~| ~t \in \mathds{R}_+ ,~ \omega \in \Gamma_t\},~~~\tau_{\tilde{\Gamma}}(\omega)=\inf \{t\geq0 |~(t,\omega) \notin \tilde{\Gamma}\},
\end{align*}
and
\begin{equation}\label{Gamma}
\Gamma:=\tilde{\Gamma} \cup~\{(t, \omega)~|~ t \geq \zeta(\omega),~ \tau_{\tilde{\Gamma}}(\omega)\geq  \zeta(\omega)\}.
\end{equation}

Then $I_{\Gamma}(t, \cdot), t\geq 0,$ is a decreasing right continuous $\{\mathcal{M}_t\}$-adapted process.

 Let $\tau_{\Gamma}(\omega):=\inf\{t\geq0~|~ (t, \omega) \notin \Gamma\}$ and $\Lambda_{\Gamma}:=\{\omega~|~\tau_{\Gamma}(\omega)\geq \zeta(\omega)\}, $ then by  (\ref{Gamma}) we have $\Lambda_{\Gamma}=\{\omega~|~\tau_{\Gamma}(\omega)=\infty\}.$  Moreover, comparing  (\ref{Lamdah}) and (\ref{Gamma}), we get $\Lambda_{\Gamma}\supset \Lambda^h$ and hence  $P^h_{\nu}((\Lambda_{\Gamma})^c)=P^h_{\nu}(\{\tau_{\Gamma}< \zeta\})=0$ for all  $\nu \in \mathcal{S}_0.$   Consequently,
\begin{align*}\label{}
\mathbb{Q}_{\nu}(\Gamma^c):&=\int^{+\infty}_0\overline{Q}_{\nu,t}
(I_{\Gamma^c}(t,\,\cdot\,))dt\\
&=E^h_{h\cdot\nu}[I_{\Lambda_{\Gamma}}\int^{+\infty}_0 I_{\{t\geq \tau_{\Gamma}\}}\frac{I_{(t<\zeta)}}{h(X_t)}dt]=0,
 ~~\forall~\nu\in\mathcal{S}_0.
\end{align*}
  We define
\begin{equation}\label{def-tildeA}
\tilde{A}_{t}(\omega):=\left\{
\begin{array}{ll}
\lim\limits_{i\to+\infty}\tilde{A}^h_{k_i}(t,\omega), &0\leq t<\tau_{\tilde{\Gamma}}(\omega),\\
\tilde{A}_{\tau_{\tilde{\Gamma}}}(\omega)_{-}, & t\geq \tau_{\tilde{\Gamma}}(\omega)>0,\\
0,      & \tau_{\Gamma}(\omega)=0 ,
\end{array}
\right.
\end{equation}
and
\begin{equation}\label{def-A}
A_{t}:=
\int^t_0e^{(\beta-\alpha)u}h(X_u)d\tilde{A}_u, ~~\forall ~ t\geq 0.
\end{equation}
Then we can check that $A$ is an $\mathcal{O}$-PCAF with defining set $\Gamma.$  Note that for $\omega \in \Lambda^h,$ we have
\begin{equation}\label{relationAA^h}
A_{t}(\omega)=\int^t_0h(X_u)dA^h_u(\omega).
\end{equation}
 For any $\beta>\alpha$ and $\nu\in \mathcal{S}_0,$  by
(\ref{hrevuz}) and Lemma \ref{lemma-potential measure between h trans} we get,
\begin{align*}
E^h_{h\cdot\nu}[\int^{+\infty}_0&\frac{e^{-(\beta-\alpha)t}I_{(t<\zeta)}}
{h(X_t)}dA_t]
=E^h_{h\cdot\nu}[I_{\Lambda^h}\int^{+\infty}_0e^{-(\beta-\alpha)t}dA^h_t]\\
=&\langle U^h_{\beta-\alpha}(h\mu),~ h\cdot\nu\rangle=\langle U_{\beta}\mu, ~\nu\rangle .
\end{align*}
Hence $A$ satisfies (\ref{RevuzS0}).

Uniqueness.

Suppose that $B$ is another $\mathcal{O}$-PCAF  with defining set $\Gamma'$ satisfying (\ref{RevuzS0}). We define
\begin{equation}\label{eq: new add function 1}
B^{h}_t(\omega)=\left\{
\begin{array}{ll}
\int^t_0\frac{1}{h(X_s)}dB_s(\omega), & \omega\in \Lambda_{\Gamma'}\\
0, & \omega\notin \Lambda_{\Gamma'}
\end{array}
\right.
\end{equation}
Then $B^{h}$ is a classical PCAF of the $h$-associated process $\mathbf{M}^h.$  By (\ref{RevuzS0}) we have for any $\beta>\alpha$ and $\nu \in \mathcal{S}_0,$
$$\langle U_{\beta}\mu,~\nu\rangle=E^h_{h\cdot \nu}[\int^{+\infty}_0\frac{e^{-(\beta-\alpha)t}I_{(t<\zeta)}}{h(X_t)}dB_t]
=E^h_{h\cdot\nu}[\int^{+\infty}_0e^{-(\beta-\alpha)t}1(X_t)dB^h_t].$$
Therefore, for any $\beta>\alpha$, by Lemma \ref{lemma-potential measure between h trans} it holds that $$E^h_x[\int^{+\infty}_0e^{-(\beta-\alpha)t}1(X_t)dB^h_t]=U^h_{\beta-\alpha}(h\cdot \mu)(x),~ q.e. ~x\in E,$$
 Hence, by \cite[Theorem 4.1.10 ]{Oshima13}  $B^h$ and $A^h$ are equivalent in the meaning of \cite[Section 5.1]{FOT11} or  \cite[Section 4.1]{Oshima13}, which means that
there exists a defining set $\tilde{\Lambda}^h$ such that  $P^h_{\nu}((\tilde{\Lambda}^h)^c)=0$ for all $\nu\in \mathcal{S}_0,$ and  $B^h_t(\omega)=A^h_t(\omega)$ for all
$t\geq0,\omega\in\tilde{\Lambda}^h$. We now set
\begin{equation*}
\tilde{\Gamma}_t:=\bigcap_{n\geq 1}\bigcup_{k\geq n} \{\omega \in \Omega~|~  B_s=A_s
~ \forall~ s\leq t+\frac{1}{k}\},
\end{equation*}
then $\tilde{\Gamma}_t \supset \tilde{\Lambda}^h,$  which implies that if we define
$$\tilde{\Gamma}:= \{(t,\omega)~| ~t \in \mathds{R}_+ ,~ \omega \in \tilde{\Gamma}_t ~\}\cap \Gamma \cap \Gamma',~~~\tau_{\tilde{\Gamma}}(\omega)=\inf \{t\geq0|~(t,\omega) \notin \tilde{\Gamma}\},$$
and
$$\Gamma'':=\tilde{\Gamma}\cup\{(t,\omega)|~t\geq\zeta(\omega),~~\tau_{\tilde{\Gamma}}(\omega)\geq\zeta(\omega)\},$$
 then $\mathbb{Q}_{\nu}((\Gamma'')^c)=0$ for all $\nu \in \mathcal{S}_0.$
We can check that $\Gamma''$ is a common defining set for $B$ and $A$, and $B$ and $A$ are identical on $\Gamma''.$  That is,  $B$ and $A$ are $\mathcal{O}$-equivalent.
\end{proof}
\begin{lemma}\label{lem-RevizS0}
Let $A$ be an $\mathcal{O}$-PCAF and $\mu\in \mathcal{S}_0.$  Then the following two assertions are equivalent to each other.

(i) A and $\mu$ satisfy (\ref{RevuzS0}).

(ii)  $\mu$ is the Revuz measure of $A,$ i.e., the assertion (\ref{eq: revuz cor}) is true.
\end{lemma}
\begin{proof}
let $A^h$ be defined by (\ref{AAh}). Then $\mu\in \mathcal{S}_0$ is the Revuz measure of $A$ w.r.t. $(\mathcal{E}, D(\mathcal{E}))$ if and only if $h\cdot \mu$ is the Revuz measure of $A^h$ w.r.t. $(\mathcal{E}^h, D(\mathcal{E}^h)).$ By the theory of classic Revuz correspondence, the latter is true if and only if for all $\beta> \alpha,$
\begin{equation}\label{hrevuzagain}
U^h_{\beta-\alpha}(h\cdot \mu)=E^h_x[\int^{+\infty}_0e^{-(\beta-\alpha)t}
1(X_t)dA^h_t], \,\,\mbox{q.e.}\,\, x\in E.
\end{equation}
Applying Lemma \ref{lemma-potential measure between h trans}, we see that (\ref{hrevuzagain}) is true if and only if (\ref{RevuzS0}) is true.
\end{proof}

\begin{proof} (Proof of Theorem \ref{thm-revuz cor}(ii))

For $\mu\in\mathcal{S}$, by Lemma \ref{SS0compact} there exists an $\mathcal{E}$-nest $(K_n)_{n\geq1}$ consisting of compact sets such that $I_{K_n}\cdot\mu\in\mathcal{S}_0$ for each $n\geq1$.
Then by Lemma \ref{them-RevuzS0} and Lemma \ref{lem-RevizS0}, there exists an unique $\mathcal{O}$-PCAF $A_{I_{K_n}\cdot\mu}$ whose Revuz measure is $I_{K_n}\cdot\mu$
for each $n\geq1$. Hence, for any $\gamma$-coexcessive $(\gamma\geq\alpha)$ function $g\in D(\mathcal{E})$ and bounded  $f\in \mathcal{B}(E)^+$ we get
\begin{align*}
&\lim\limits_{\beta\to+\infty}\beta(g,U^{\beta+\gamma}_{I_{K_n}\cdot A_{I_{K_{n+1}\cdot\mu}}}f)_m\\
&=\lim\limits_{\beta\to+\infty}\beta(g,hE^h[\int^{+\infty}_0e^{-(\beta+\gamma-\alpha)t}\frac{fI_{K_n}}{h}(X_t)I_{(t<\zeta)}dA_{I_{K_{n+1}}\cdot\mu}(t)])_m\\
&=\int_EfgI_{K_n}d(I_{K_{n+1}}\cdot\mu)=\int_EfgI_{K_n}d\mu,
\end{align*}
which means the Revuz measure of $I_{K_n}\cdot A_{I_{K_{n+1}\cdot\mu}}$ is also $I_{K_n}\cdot \mu$.
Hence, $I_{K_n}\cdot A_{I_{K_{n+1}}\cdot\mu}$ and $A_{I_{K_n}\cdot\mu}$ are $\mathcal{O}$-equivalent. ( Here and henceforth $I_{K_n}\cdot A_{I_{K_{n+1}}\cdot\mu}(t):=\int_0^tI_{K_n}(X_u) dA_{I_{K_{n+1}}\cdot\mu}(u).$)

Let $\Gamma^1$ be a defining set of $A_{I_{K_1}\cdot\mu}.$  For each $n\geq 2,$ we may take a common defining set
$\Gamma^n$ such that $I_{K_{n-1}}\cdot A_{I_{K_{n}}\cdot\mu}$ and $A_{I_{K_{n-1}}\cdot\mu}$ are identical on $\Gamma^n.$  Without loss of generality we may assume the $\Gamma^n \subset \Gamma^{n-1}$ for each $n\geq 2.$  Let $\tau_{\Gamma^n}(\omega)=\inf\{t\geq0~|~ (t, \omega) \notin \Gamma^n\}$ and  $\tau_{\Gamma^{\infty}}(\omega):=\inf_{n\geq 1}\tau_{\Gamma^n}(\omega).$  Then by the right continuity of $\{\mathcal{M}_t\},$ we have $\tau_{\Gamma^{\infty}}\in \mathcal{T}.$ Moreover, since
$\{\tau_{\Gamma^{\infty}}< \zeta\} \subset \cup_{n\geq 1} \{\tau_{\Gamma^n}<\zeta\},$ therefore $P^h_{\nu}(\{\tau_{\Gamma^{\infty}}< \zeta\})=0$ for all $\nu\in\mathcal{S}_0.$

We define
$\tau_{K_n}(\omega)=\inf\{t\geq 0~|~ X_t(\omega) \notin K_n\}$
for $n\geq 1.$ Similar to the argument of  \cite[IV.Proposition 5.30 (i)]{MR92}, we can show that  there exists an $\mathcal{E}$-exceptional set $N$, such that for any $x\in E\setminus N$, $P_x^{h}(\lim_{n\rightarrow \infty}\tau_{K_n}<\zeta)=0$. Therefore, if we let $\tau_{K_{\infty}}:=\lim_{n\rightarrow \infty}\tau_{K_n},$ then $P^h_{\nu}(\{\tau_{K_{\infty}}< \zeta\})=0$ for all $\nu\in\mathcal{S}_0.$
We now set $\eta=\tau_{\Gamma^{\infty}}\wedge \tau_{K_{\infty}}$ and define
\begin{equation*}
\Gamma:= \{(t,\omega)~|~ t< \eta(\omega)\}\cup \{(t,\omega)~|~ t\geq \zeta (\omega),~ \eta(\omega)\geq \zeta (\omega)\}.
\end{equation*}
We can check that $\Gamma$ satisfies Definition \ref{def-PCAF2} (a),(b), and (c). Moreover,  if $t< \eta(\omega),$  then
$A_{I_{K_{n}}\cdot\mu}(t,\omega)= I_{K_{n}}\cdot A_{I_{K_{n+l}}\cdot\mu}(t,\omega)$ for all $n,l\geq 1.$ Set $\tau_{K_0}(\omega):=0$ and define
\begin{equation}\label{def-A for smooth}
A(t,\omega)=\left\{
\begin{array}{ll}
A_{I_{K_n}\cdot\mu}(t,\omega), &0\leq t<\eta(\omega),\tau_ {K_{n-1}}(\omega)\leq t<\tau_{K_n}(\omega)\\
A_{\eta-}(\omega), & t\geq\eta(\omega)>0\\
0,      & \eta(\omega)=0 .
\end{array}
\right.
\end{equation}
Then we can check that $A$ satisfies Definition \ref{def-PCAF2} (d) and (e). Therefore, $A$ is an $\mathcal{O}$-PCAF. By Part (i) of Theorem \ref{thm-revuz cor}, there exists an unique smooth measure $\mu_A,$  such that for any  $\gamma$-coexcessive ($\gamma>\alpha$) function $g\in D(\mathcal{E})$  and any bounded function $f\in \mathcal{B}(E)^+,$ it holds that
\begin{equation*}
\lim\limits_{\beta\to+\infty}\beta(g,U^{\beta+\gamma}_Af)_m=
\int_Efg\mu_A(dx).
\end{equation*}
But for any $n\geq 1,$ we have
\begin{align*}
&\int_EfgI_{K_n}\mu_A(dx)=\lim\limits_{\beta\to+\infty}
\beta(g,U^{\beta+\gamma}_A
fI_{K_n})_m\\
&=\lim\limits_{\beta\to+\infty}\lim\limits_{l\to+\infty}\beta(g,
hE_{\cdot}^h[I_{\{\eta\geq \zeta\}}\int^{\tau_{K_{n+l}}}_0e^{-(\beta-\alpha)t}
\frac{fI_{K_n}}{h}(X_t)I_{(t<\zeta)}dA_t])_m\\
&=\lim\limits_{\beta\to+\infty}\lim\limits_{l\to+\infty}\beta(g,
hE_{\cdot}^h[\int^{\tau_{K_{n+l}}}_0e^{-(\beta-\alpha)t}
\frac{f}{h}(X_t)I_{(t<\zeta)}dA_{I_{K_{n}}\cdot\mu}(t)])_m\\
&=\int_EfgI_{K_n}\mu(dx).
\end{align*}
Therefore, $\mu_A=\mu,$ i.e., $\mu$ is the Revuz measure of $A.$ Suppose that $B$ is another $\mathcal{O}$-PCAF whose Revuz measure is $\mu.$ Then $I_{K_n}\cdot B$ is Revuz corresponding to $I_{K_n}\cdot \mu$ for all $n\geq 1.$ Hence $B$ is $\mathcal{O}$-equivalent to $A.$

\end{proof}

\subsection{Optional measure $\mathbb{Q}_{x}^{A}(\cdot)$ generated by A}
\label{independence}
In this subsection we are going to show that the
the Revuz correspondence defined by (\ref{eq: revuz cor}) and (\ref{UbetaA}) is independent of $h\in \mathcal{H}$ (see Corollary \ref{Revuzindependence}).
To this end we introduce an optional measure $\mathbb{Q}_{x}^{A}(\cdot)$ generated by an $\mathcal{O}$-PCAF $A$, which we believe will have interest by its own and
will be useful in the further study of stochastic analysis related to positivity preserving coercive forms.

\begin{definition}\label{optionalmeasure}
Let $A$ be an $\mathcal{O}$-PCAF.  For $x\in E,$  we define a $\sigma$-finite measure $\mathbb{Q}_{x}^{A}(\cdot)$ on $\mathcal{O}$ by setting:
\begin{equation}\label{QxA}
\mathbb{Q}_{x}^{A}(H):=h(x)E_x^h[\int_0^{\infty}I_H(t,\cdot)\frac{e^{\alpha t} I_{\{t<\zeta\}}}{h(X_t)}dA_t], ~~\forall ~H\in \mathcal{O}.
\end{equation}
We call $\mathbb{Q}_{x}^{A}(\cdot)$ an optional measure generated by $A.$

For $\nu \in \mathcal{S}_0,$ we write
\begin{equation}\label{QnuA}
\mathbb{Q}_{\nu}^{A}(H):=\int_E \mathbb{Q}_{x}^{A}(H)\nu(dx)
=E_{h\cdot\nu}^h[\int_0^{\infty}I_H(t,\cdot)\frac{e^{\alpha t} I_{\{t<\zeta\}}}{h(X_t)}dA_t], ~~\forall ~H\in \mathcal{O}.
\end{equation}

\end{definition}

\begin{lemma}\label{QxB}
  Let  $B$ be an $\mathcal{O}$-PCAF which is $\mathcal{O}$-equivalent to $A,$  and $\mathbb{Q}_{x}^{B}(\cdot)$ be defined by (\ref{QxA}) with $A$ replaced by $B.$ Then there exists an $\mathcal{E}$-exceptional set $N,$ such that $\mathbb{Q}_{x}^{B}(\cdot)=\mathbb{Q}_{x}^{A}(\cdot)$ for all $x\in E\setminus N,$ and hence  $\mathbb{Q}_{\nu}^{B}(\cdot)=\mathbb{Q}_{\nu}^{A}(\cdot)$ for all $\nu\in\mathcal{S}_0.$
 \end{lemma}
\begin{proof}
Let $\Gamma$ be a common defining set for $B$ and $A,$ and $\Lambda_{\Gamma}$ be specified by Definition \ref{def-PCAF2} (c). Then there exists an $\mathcal{E}$-exceptional set $N$ such that $P^h_x((\Lambda_{\Gamma})^c)=E^h_x[I_{\{ \tau_{\Gamma}<\zeta\}}]=0$ for $x\in E\setminus N$ (cf. the proof of Proposition \ref{classic}). Then,
for any $H\in \mathcal{O},$ we have
\begin{align*}
\mathbb{Q}_{x}^{B}(H)&=h(x)E_x^h[I_{\Lambda_{\Gamma}}
\int_0^{\infty}I_H(t,\cdot)\frac{e^{\alpha t} I_{\{t<\zeta\}}}{h(X_t)}dB_t]\\
&=h(x)E_x^h[I_{\Lambda_{\Gamma}}
\int_0^{\infty}I_H(t,\cdot)\frac{e^{\alpha t} I_{\{t<\zeta\}}}
{h(X_t)}dA_t]=\mathbb{Q}_{x}^{A}(H).
\end{align*}
\end{proof}
In the proof of Theorem \ref{Qindependenth} below we shall make use of the predictable $\sigma$-field $\mathcal{P}$ related to $\{\mathcal{M}_t\}$. Recall that $\mathcal{P}$ is the $\sigma$-field on $[[0, \infty[[:=\mathds{R}_+\times \Omega$ generated by all the left continuous $\{\mathcal{M}_t\}$-adapted processes. It is known that (cf. \cite[Theorem 3.21]{HWY92}) $\mathcal{P} \subset \mathcal{O}$ and $\mathcal{P}$ is the $\sigma$-field generated by the following sets
\begin{equation}\label{mathcalP}
\{\{0\}\times F~|~ F\in \mathcal{M}_{0}\}\bigcup \{[p,q)\times F~|~ 0< p<q<\infty,~ p,q\in \mathds{Q}_+, ~F\in \mathcal{M}_{p-}\},
\end{equation}
where $\mathds{Q}_+$ stands for all the nonnegative rational numbers and $\mathcal{M}_{p-}= \bigvee_{s<p}\mathcal{M}_{s}.$

Recall that $\mathcal{H}$ is the collection of all strictly positive $\mathcal{E}$-quasi-continuous $\alpha$-excessive functions.
\begin{theorem}\label{Qindependenth}
 The optional measure  $\mathbb{Q}_{x}^{A}(\cdot)$ defined by (\ref{QxA}) is independent of $h\in \mathcal{H}$ in the following sense.

(i) Let $\mathbb{Q}_{x}^{'A}(\cdot)$ be defined by (\ref{QxA}) with $h$  replaced by another $h'\in \mathcal{H}.$  Then there exists an $\mathcal{E}$-exceptional set $N$ such that $\mathbb{Q}_{x}^{'A}(\cdot)=\mathbb{Q}_{x}^{A}(\cdot)$ for all $x\in E\setminus N.$

(ii) Consequently, let $\mathbb{Q}_{\nu}^{'A}(\cdot)$ be defined by (\ref{QnuA}) with $h$ replaced by another $h'\in \mathcal{H}.$  Then for any $\nu \in \mathcal{S}_0,$ we have $\mathbb{Q}_{\nu}^{'A}(\cdot)=\mathbb{Q}_{\nu}^{A}(\cdot).$
\end{theorem}
\begin{proof}
Let $A$ be an $\mathcal{O}$-PCAF. Then by Theorem \ref{thm-revuz cor} there exists $\mu \in \mathcal{S}$ satisfying (\ref{eq: revuz cor}). We assume first that $\mu\in \mathcal{S}_0.$ Then by Lemma \ref{lem-RevizS0} $A$ and $\mu$ satisfy (\ref{RevuzS0}) w.r.t. $h$. Suppose that $h'$ is another strictly positive $\mathcal{E}$-quasi-continuous $\alpha$-excessive function. Following the procedure of Lemma \ref{them-RevuzS0} we may construct another $\mathcal{O}$-PCAF $A'$ such that $A'$ and $\mu$ satisfy (\ref{RevuzS0}) w.r.t. $h'$.  More precisely,  let $g_n$ be the same as specified by (\ref{gn}), then
(\ref{Rgn}) holds also true when $h$ is replaced by $h'$. Moreover, we can choose the same subsequence $\{n_l\}$ such that $(R^{h'}_{\beta-\alpha}\frac{f_k}{h'})_{k\in\mathds{N}}$ with $f_k:=\frac{1}{k}\sum_{l=1}^{k}g_{n_l}$ converges  strongly in $(\mathcal{E}^{h'},D(\mathcal{E}^{h'}))$,
and take the same subsequence $k_i$ such that
\begin{align*}
 \mathcal{E}^{h'}_{\beta-\alpha}&(R^{h'}_{\beta-\alpha}\frac{f_{k_{i+1}}-f_{k_{i}}}
 {h'}, R^{h'}_{\beta-\alpha}\frac{f_{k_{i+1}}-f_{k_{i}}}{h'})= \\ & \mathcal{E}^h_{\beta-\alpha}(R^h_{\beta-\alpha}\frac{f_{k_{i+1}}-f_{k_{i}}}{h},
 R^h_{\beta-\alpha}\frac{f_{k_{i+1}}-f_{k_{i}}}{h})<2^{-6i}.
\end{align*}
Repeating the construction in Lemma \ref{them-RevuzS0}, we denote
\begin{equation}\label{A'htk}
\tilde{A'}^{h'}_k(t,\omega):=\int^t_0e^{-(\beta-\alpha)s}\frac{f_k}{h'}(X_s)ds.
\end{equation}
Define
\begin{align*}\label{Gammat}
\Gamma'_t=\bigcap_{m\geq 1}\bigcup_{l\geq m} \{\omega \in \Omega~|~  \tilde{A'}^{h'}_{k_i}(u,\omega) \ &\mbox{converges uniformly}\\
&\mbox{for} \ u\in [0,t+\frac{1}{l}]~ \mbox{when}~ i \rightarrow \infty\},
\end{align*}
\begin{align*}
\tilde{\Gamma'}= \{(t,\omega)~| ~t \in \mathds{R}_+ ,~ \omega \in \Gamma'_t\},~~~\tau_{\tilde{\Gamma'}}(\omega)=\inf \{t\geq0|~(t,\omega) \notin \tilde{\Gamma'}\},
\end{align*}
and
\begin{equation*}
\Gamma':=\tilde{\Gamma'} \cup~\{(t, \omega)~|~ t \geq \zeta(\omega),~ \tau_{\tilde{\Gamma'}}(\omega)\geq  \zeta(\omega)\}.
\end{equation*}

Let $\tau_{\Gamma'}(\omega):=\inf\{t\geq0~|~ (t, \omega) \notin \Gamma'\}$
and define
\begin{equation}\label{def-tildeA'}
\tilde{A}'_{t}(\omega):=\left\{
\begin{array}{ll}
\lim\limits_{i\to+\infty}\tilde{A'}^{h'}_{k_i}(t,\omega), &0\leq t<\tau_{\tilde{\Gamma'}}(\omega),\\
\tilde{A'}_{\tau_{\tilde{\Gamma'}}}(\omega)_{-}, & t\geq \tau_{\tilde{\Gamma'}}(\omega)>0,\\
0,      & \tau_{\Gamma'}(\omega)=0 .
\end{array}
\right.
\end{equation}
Define further
\begin{equation}\label{def-A'}
A'_{t}:=
\int^t_0e^{(\beta-\alpha)u}h'(X_u)d\tilde{A}'_u, ~~\forall ~ t\geq 0.
\end{equation}
Then $A'$ is an $\mathcal{O}$-PCAF with defining set $\Gamma'$ satisfying (\ref{RevuzS0}) with $h$ replaced by $h'$.

We claim that $A'$ and $A$ are in fact
$\mathcal{O}$-equivalent. To verify this we define  $\tilde{\Gamma}=\Gamma\cap\Gamma'$
 and $\tau_{\tilde{\Gamma}}(\omega):=\inf\{t\geq0~|~ (t, \omega) \notin \tilde{\Gamma}\}.$  Let $t<\tau_{\tilde{\Gamma}}(\omega)$, then both  $\tilde{A'}^{h'}_{k_i}(s,\omega)$ and $\tilde{A}^h_{k_i}(s,\omega)$ converge uniformly for $s \in [0,t]$. Comparing (\ref{Ahtk}) and (\ref{A'htk}), we have $$\tilde{A'}^{h'}_{k_i}(s,\omega)=\int^s_0\frac{h}{h'}(X_u)d\tilde{A}^h_{k_i}
(u,\omega),~~\forall~s<\tau_{\tilde{\Gamma}}(\omega).$$
Consequently, for any $T>0$ we have
  $$\int^t_0h'(X_u)I_{\{h(X_u)<T,~ h'(X_u)< T\}}d\tilde{A}_u'(\omega)
=\int^t_0h(X_u)I_{\{h(X_u)<T,~ h'(X_u)< T\}}d\tilde{A}_u(\omega).$$
Letting $T\rightarrow \infty,$ we get $A'_t(\omega)=A_t(\omega)$ for  $t<\tau_{\tilde{\Gamma}}(\omega).$ Therefore, $A'$ and $A$ are $\mathcal{O}$-equivalent.

Let $\mathbb{Q}_{x}^{'A'}(\cdot)$ be the optional measure constructed with  $A'$ and $h'$. We are going to show that $\mathbb{Q}_{x}^{'A'}(\cdot)$ and $\mathbb{Q}_{x}^{A}(\cdot)$ are identical. To this end we  take an  $\mathcal{E}$-nest $(F_k)_{k\geq 1}$  consisting of compact sets such that $h\in C(\{F_k\})$ and  $h'\in C(\{F_k\}).$  Let $\tau_{F_k}=\inf\{t\geq 0~|~ X_t\notin F_k\}$ and $\tau_k=\inf\{t\geq0: A_t> k ~\mbox{or}~ A'_t> k\}.$ Define $T_k=\tau_{F_k}\wedge\tau_k\wedge\tau_{\Gamma}\wedge\tau_{\Gamma'}
\wedge k, $ where $\tau_{\Gamma}$  and $\tau_{\Gamma'}$ are specified by  Definition \ref{def-PCAF2} (c).   Then, because both $h$ and $h'$ are strictly positive and finite on each compact set $F_k,$ and $\nu(F_k)<\infty$ for all $\nu\in \mathcal{S}_0$ and $k\geq 1,$ we have $\mathbb{Q}_{\nu}^{'A'}([[0,~T_k[[)<\infty$ and $\mathbb{Q}_{\nu}^{A}([[0,~T_k[[)<\infty$ for all $\nu\in \mathcal{S}_0$ and $k\geq 1.$

Suppose that $H=[s,t)\times F$ with $F\in \mathcal{M}_{s-}$ and $s,t\in \mathds{Q}_+,~0<s<t<\infty.$  For fixed $\nu\in \mathcal{S}_0$ and fixed $k\geq 1,$ we have by (\ref{def-A}),(\ref{def-tildeA}),(\ref{Ahtk}) and (\ref{eq: relation Mmu and F}),
\begin{align*}
\mathbb{Q}^{A}_{\nu}(H\cap[[0,T_k[[)&=E_{h\cdot\nu}^h[\int_0^{\infty}I_F
I_{[[s\wedge T_k,t\wedge T_k))}(u)\frac{e^{\alpha u} I_{\{u<\zeta\}}}{h(X_u)}dA_u]\\
&=E_{h\cdot\nu}^h[\int_0^{\infty}I_F
I_{[[s\wedge T_k,t\wedge T_k))}(u)e^{\beta u} I_{\{u<\zeta\}}d\tilde{A}_u]\\
&=\lim_{i\rightarrow \infty}E_{h\cdot\nu}^h[\int_0^{\infty}I_F
I_{[[s\wedge T_k,t\wedge T_k))}(u)e^{\alpha u} I_{\{u<\zeta\}}\frac{f_{k_i}(X_u)}{h(X_u)}du]\\
&=\lim_{i\rightarrow \infty}\int^{+\infty}_0\overline{Q}_{\nu,u}(I_F
I_{[[s\wedge T_k,t\wedge T_k))}(u)f_{k_i}(X_u))du.
\end{align*}
On the other hand, by (\ref{def-A'}),(\ref{def-tildeA'}),(\ref{A'htk}) and (\ref{eq: relation Mmu and F}), we get
\begin{align*}
\mathbb{Q}^{'A'}_{\nu}(H\cap[[0,T_k[[)&=E_{h'\cdot\nu}^{h'}[\int_0^{\infty}I_F
I_{[[s\wedge T_k,t\wedge T_k))}(u)\frac{e^{\alpha u} I_{\{u<\zeta\}}}{h'(X_u)}dA'_u]\\
&=E_{h'\cdot\nu}^{h'}[\int_0^{\infty}I_F
I_{[[s\wedge T_k,t\wedge T_k))}(u)e^{\beta u} I_{\{u<\zeta\}}d\tilde{A'}_u]\\
&=\lim_{i\rightarrow \infty}E_{h'\cdot\nu}^{h'}[\int_0^{\infty}I_F
I_{[[s\wedge T_k,t\wedge T_k))}(u)e^{\alpha u} I_{\{u<\zeta\}}\frac{f_{k_i}(X_u)}{h'(X_u)}du]\\
&=\lim_{i\rightarrow \infty}\int^{+\infty}_0\overline{Q}_{\nu,u}(I_F
I_{[[s\wedge T_k,t\wedge T_k))}(u)f_{k_i}(X_u))du.
\end{align*}
Therefore, for all $F\in \mathcal{M}_{s-}$ and $s,t\in \mathds{Q}_+,~0<s<t<\infty,$ it holds that
\begin{equation}\label{stF}
\mathbb{Q}^{'A'}_{\nu}(([s,t)\times F)\cap[[0,T_k[[)=\mathbb{Q}^{A}_{\nu}(([s,t)\times F)\cap[[0,T_k[[).
 \end{equation}
 It is trivial that
 \begin{equation}\label{0F}
 \mathbb{Q}^{'A'}_{\nu}((\{0\}\times F)\cap[[0,T_k[[)=\mathbb{Q}^{A}_{\nu}((\{0\}\times F)\cap[[0,T_k[[)=0, ~\forall~F\in \mathcal{M}_{0}.
\end{equation}

Applying monotone class argument, by (\ref{mathcalP}) we get $\mathbb{Q}^{'A'}_{\nu}(H\cap[[0,T_k[[)=\mathbb{Q}^{A}_{\nu}(H\cap[[0,T_k[[)$ for all $H\in \mathcal{P}.$
Therefore,  for all $T\in \mathcal{T}$ we have
\begin{align*}
\mathbb{Q}^{'A'}_{\nu}([[T,\infty[[\cap[[0,T_k[[)&=
\mathbb{Q}^{'A'}_{\nu}(]]T,\infty[[\cap[[0,T_k[[)\\
&=\mathbb{Q}^{A}_{\nu}(]]T,\infty[[\cap[[0,T_k[[)=
\mathbb{Q}^{A}_{\nu}([[T,\infty[[\cap[[0,T_k[[).
\end{align*}
Applying monotone class argument again, by (\ref{defoptionalfieldall}) we get
\begin{equation}\label{H0Tk}
\mathbb{Q}^{'A'}_{\nu}(H\cap[[0,T_k[[)=\mathbb{Q}^{A}_{\nu}(H\cap[[0,T_k[[) ,~~\forall ~H\in \mathcal{O}.
\end{equation}

Denote by $T_{\infty}=\lim_{k\rightarrow \infty}T_{k}.$  Similar to the argument of  \cite[IV. Proposition 5.30 (i)]{MR92}, we can show that  there exists an $\mathcal{E}$-exceptional set $N$, such that for any $x\in E\setminus N$, $P_x^{h}(\lim_{k\rightarrow \infty}\tau_{F_k}<\zeta)=0$. By Definition \ref{def-PCAF2} (d),  we have $\lim_{k\rightarrow \infty}\tau_{k}\geq \tau_{\Gamma}\wedge \tau_{\Gamma'}\wedge \zeta.$
From the proof of Proposition \ref{classic}, we see that $E^h_x[I_{\{ \tau_{\Gamma}<\zeta\}}]=0$ and $E^{h'}_x[I_{\{ \tau_{\Gamma'}<\zeta\}}]=0$ for q.e. $x\in E.$ Making use of the above  facts, we can show that
\begin{equation}\label{T-infty}
E^h_{\nu}[I_{\{T_{\infty}< \zeta\}}]=E^{h'}_{\nu}[I_{\{T_{\infty}< \zeta\}}]=0,~~ \forall~ \nu\in \mathcal{S}_0.
\end{equation}
Therefore, for any $H\in \mathcal{O}$ we have
\begin{align}\label{limTk}
\lim_{k\rightarrow \infty}\mathbb{Q}^{A}_{\nu}(H\cap[[0,T_k[[)&=
\lim_{k\rightarrow \infty}E_{h\cdot\nu}^h[I_{\{T_{\infty} \geq \zeta\}}\int_0^{T_{k}}I_H(u,\cdot)\frac{e^{\alpha u} I_{\{u<\zeta\}}}{h(X_u)}dA_u] \nonumber\\
=&E_{h\cdot\nu}^h[I_{\{T_{\infty} \geq \zeta\}}\int_0^{\infty}I_H(u,\cdot)\frac{e^{\alpha u} I_{\{u<\zeta\}}}{h(X_u)}dA_u]=\mathbb{Q}^{A}_{\nu}(H).
\end{align}
Similarly, for any $H\in \mathcal{O}$ we have
\begin{equation}\label{limTk'}
\lim_{k\rightarrow \infty}\mathbb{Q}^{'A'}_{\nu}(H\cap[[0,T_k[[)=\mathbb{Q}^{'A'}_{\nu}(H).
\end{equation}
Consequently by (\ref{H0Tk}) we conclude that $\mathbb{Q}^{A}_{\nu}(\cdot)=\mathbb{Q}^{'A'}_{\nu}(\cdot)$ for all $\nu\in \mathcal{S}_0$.

 Note that $E$ is a Lusin space, therefore  $\mathcal{B}(E)$ is countably generated which implies that $\mathcal{F}^0_s$ is countably generated for any $s\in \mathds{Q}_+.$ Then for each $s,t\in \mathds{Q}_+,~0<s<t<\infty,$ by what we have proved we may take an $\mathcal{E}$-exceptional set $N_{st}$ such that for all $x\in E\setminus N_{st}$ and $F\in \mathcal{F}^0_r,~r<s,$ it holds that
\begin{equation}\label{xstF}
\mathbb{Q}^{'A'}_{x}(([s,t)\times F)\cap[[0,T_k[[)=\mathbb{Q}^{A}_{x}(([s,t)\times F)\cap[[0,T_k[[).
 \end{equation}
 By virtue of Lemma \ref{thm-relation M and F} (i), we see that in fact (\ref{xstF}) holds for all $F\in\mathcal{M}_r,~r<s,$ and hence it is true for all $F\in\mathcal{M}_{s-}.$  Because (\ref{0F}) is true for any $x\in E,$ therefore if we set $N_1=\bigcup_{s,t\in \mathds{Q}_+,s<t} N_{st},$ then applying twice monotone class arguments we can get
 \begin{equation}\label{xH0Tk}
\mathbb{Q}^{'A'}_{x}(H\cap[[0,T_k[[)=\mathbb{Q}^{A}_{x}(H\cap[[0,T_k[[) ,~~\forall ~H\in \mathcal{O},~x\in E\setminus N_1.
\end{equation}
By (\ref{T-infty}), we can take an $\mathcal{E}$-exceptional set $N_2$ such that
\begin{equation}\label{xT-infty}
E^h_{x}[I_{\{T_{\infty}< \zeta\}}]=E^{h'}_{x}[I_{\{T_{\infty}< \zeta\}}]=0,~~ \forall~ x\in E\setminus N_2.
\end{equation}
Following the arguments of (\ref{limTk}) and (\ref{limTk'}), we get
$\mathbb{Q}^{A}_{x}(\cdot)=\mathbb{Q}^{'A'}_{x}(\cdot)$ for all $x\in E\setminus (N_1\cup N_2).$  Because $A$ and $A'$ are $\mathcal{O}$-equivalent,  by Lemma \ref{QxB} there exists an $\mathcal{E}$-exceptional set $N_3$ such that $\mathbb{Q}^{'A}_{x}(\cdot)=\mathbb{Q}^{'A'}_{x}(\cdot)$ for all $x\in E\setminus N_3.$  We now define $N=N_1\cup N_2\cup N_3,$ then $N$ is an $\mathcal{E}$-exceptional set and $\mathbb{Q}^{A}_{x}(\cdot)=\mathbb{Q}^{'A'}_{x}(\cdot)
=\mathbb{Q}^{'A}_{x}(\cdot)$ for all $x\in E\setminus N.$ Thus the theorem is proved  in the case that $\mu \in \mathcal{S}_0$ where $\mu$ is Revuz corresponding to $A$ by (\ref{eq: revuz cor}) w.r.t. $h$.

We now extend the above results to the general situation. Let $A$ be an $\mathcal{O}$-PCAF. Suppose that $\mathbb{Q}^{A}_{x}(\cdot)$ is defined by (\ref{QxA}) with some $h\in \mathcal{H}$ and $\mathbb{Q}^{'A}_{x}(\cdot)$ is defined by (\ref{QxA}) with another $h'\in \mathcal{H}.$  Suppose that $\mu \in \mathcal{S}$ Revuz corresponding to $A$ by (\ref{eq: revuz cor}) w.r.t. $h$.  We take an $\mathcal{E}$-nest $(K_n)_{n\geq1}$ consisting of compact sets such that $I_{K_n}\cdot\mu\in\mathcal{S}_0$ for each $n\geq1$.  Then $I_{K_n}\cdot A$ is Revuz corresponding to $I_{K_n}\cdot \mu$ w.r.t. $h$.  By what we have proved, for each $n\geq 1$, there exists an $\mathcal{E}$-exceptional set $N_n$ such that $\mathbb{Q}^{'I_{K_n}\cdot A}_{x}(\cdot)=\mathbb{Q}^{I_{K_n}\cdot A}_{x}(\cdot)$ for  $x\in E\setminus N_n.$  Let $N=\cup_{n\geq 1}N_n.$  Then for $x\in E\setminus N$ we have
$$\mathbb{Q}^{'A}_{x}(H)=\lim_{n\rightarrow \infty}\mathbb{Q}^{'I_{K_n}\cdot A}_{x}(H)=\lim_{n\rightarrow \infty}\mathbb{Q}^{I_{K_n}\cdot A}_{x}(H)=\mathbb{Q}^{A}_{x}(H), ~\forall ~ H\in \mathcal{O}.$$
Therefore, Theorem \ref{Qindependenth}  is true in general case.
\end{proof}

\begin{corollary}\label{Revuzindependence}
The Revuz correspondence specified by (i) and (ii) of Theorem \ref{thm-revuz cor} is independent of $h\in \mathcal{H}.$
\end{corollary}
\begin{proof}
We denote by $U^{'\beta}_A$ the formula (\ref{UbetaA}) with $h$ replaced by $h'$. For any $f\in \mathcal{B}(E)^+$, if we set $H(t) = e^{-\beta t}f(X_t),$
then,
\begin{align*}
U^{\beta}_Af(x)&=h(x)E^h_x[\int^{+\infty}_0e^{-(\beta-\alpha)t}
\frac{f}{h}(X_t)I_{(t<\zeta)}dA_t]\\
&=\mathbb{Q}^{A}_{x}(H)=\mathbb{Q}^{'A}_{x}(H)=U^{'\beta}_Af(x) ~q.e. \  x\in E.
\end{align*}
Therefore, the corollary is true.
\end{proof}

\section*{Akonowledgments}
We are indebted to M. Fukushima who brought our attention to the work
 of pseudo Hunt processes introduced in \cite{Oshima13}, which stimulated this research. We are grateful to Y. Oshima who sent us his manuscript \cite{Oshima13} which helps this research. We thank Mufa Chen, Zhenqing Chen and Michael Roeckner for their comments and discussions. The initial idea of this research was mentioned at the Workshop on Probability Theory with Applications (Dec 19-21, 2015, Macau), and the main results of this research were announced at the International Conference on Stochastic Partial Differential Equations and Related Fields (October 10-14, 2016, Bielefeld), we would like to thank the organizers of the above two stimulating conferences for their kind invitation and hospitality.


\begin{thebibliography}{99}
\bibitem[An74]{An74}
\textsc{Ancona, A.} (1974). Theorie du potentiel dans les espaces fonctionnels $\grave{a}$ forme coercive, \textit{Cours 3erne cycle -Paris VI.}

\bibitem[An76]{An76}
\textsc{Ancona, A.} (1976). Continuit$\acute{e}$ des contractions dans les espaces de Dirichlet, \textit{C. R. Acad. Sci. Paris S$\acute{e}$r. A}. \textbf{282} 871-873. MR{0415291}

\bibitem[Ar08]{Ar08}
\textsc{Arendt, W.} (2008). Positive semigroups of kernel operators, \textit{Potsitivity}. \textbf{12} 25-44. MR{2373131}

\bibitem[BCR16]{BCR16}
\textsc{Beznea, L., Cimpean, I.} and \textsc{Roeckner, M.} (2016). A new approach to the existence of invariant measures for Markovian semigroups, \textit{preprint}. arXiv:1508.06863v3.

\bibitem[BD58]{BD58} \textsc{Beurling, A.} and \textsc{Deny, J.} (1958). Espaces de Dirichlet, \textit{Acta Math.} \textbf{99}, 203-224. MR{0098924}

\bibitem[BG68]{BG_markov_and_potential}  \textsc{Blumenthal, R. M.} and \textsc{Getoor, R. K.} (1968). \textit{Markov Processes and Potential Theory}. Academic Press, New York and London. MR{0264757}


\bibitem[Bl71]{Bl71}
\textsc{Bliedtner, J.} (1971). Dirichlet forms on regular functional spaces, \textit{Leture Notes in Mathematics}. \textbf{226} 15-62, Springer, Berlin. MR{0399490}

\bibitem[CF12]{Chen_Fukushima}
\textsc{Chen, Z. Q.} and \textsc{Fukushima, M.} (2012).
\textit{Symmetric Markov Processes, Time Change, and Boundary
  Theory}.
Princeton University Press, Princeton. MR{2849840}



\bibitem[De15]{De15}
\textsc{Dell¡¯Antonio, G.} (2015).  \textit{Lectures on the Mathematics of Quantum Mechanics I}. Atlantis Press, Paris. MR{3309262}

\bibitem[De16]{De16}
\textsc{Dell¡¯Antonio, G.} (2016). \textit{Lectures on the Mathematics of Quantum Mechanics II}. Atlantis Press, Paris. MR{3496709}

\bibitem[Fe52]{Fe52}
\textsc{Feller, W.} (1952). On positivity preserving semigroups of transformations on $C[r_1,r_2]$,
\textit{Ann. Soc. Polon. Math.}
\textbf{25} 85-94. MR{0055573}

\bibitem[Fe54]{Fe54}
\textsc{Feller, W.} (1954). The general diffusion operator and positivity preserving semi-groups
in one dimension,
\textit{Ann. of Math.}
\textbf{60} 417-436. MR{0065809}

\bibitem[Fi01]{Fi01}
\textsc{Fitzsimmons, P. Z.} (2011).
On the quasi-regularity of semi-Dirichlet forms,
\textit{Potential Anal.}
\textbf{15} 158-185. MR{1837263}

\bibitem[FOT11]{FOT11}
\textsc{Fukushima, M., Oshima, Y.} and \textsc{Takeda, M.} (2011).
\textit{Dirichlet Forms and Symmetric Markov Processes}. Second revised and extended edition.
Walter de Gruyter $\&$ Co., Berlin. MR{2778606}


\bibitem[GMNO15]{GMNO15}
\textsc{Gesztesy, F., Mitrea, M., Nichols, R.}, and \textsc{Ouhabaz E. M.} (2015). Heat kernel bounds for elliptic partial
differential operators in divergence form with Robin-type boundary conditions II, \textit{Proc. Amer. Math. Soc.} \textbf{143}  1635-1649. MR{3314076}

\bibitem[HMS11]{HMS11}
\textsc{Han, X. F., Ma, Z. M.} and \textsc{Sun, W.} (2011).
h\^{h}-transforms of positivity preserving semigroups and associated
  Markov processes.
\textit{Acta Math. Sin. (Engl. Ser.)}, \textbf{27} 369-376. MR{2754041}

\bibitem[HWY92]{HWY92}
\textsc{He, S., }\textsc{Wang, J.} and \textsc{Yan, J.} (1992).
\textit{Semimartingale theory and stochastic calculus}.
Science Press, Beijing; CRC Press, Boca Raton. MR{1219534}

\bibitem[Hi00]{Hi00}
\textsc{Hino, M.} (2000). Exponential decay of positivity preserving semigroups on Lp, \textit{Osaka J. Math}. \textbf{37}  603-624. MR{1789439}

\bibitem[KL75]{KL75} \textsc{Klein, A.} and  \textsc{Landau, L.J.} (1975).
Singular Perturbations of Positivity Preserving
Semigroups via Path Space Techniques, \textit{J. Funct. Anal.}. \textbf{20} 44-82. MR{0381580}

\bibitem[MOR95]{MOR95}
\textsc{Ma, Z. M., }\textsc{Overbeck, L.} and \textsc{R\"{o}ckner, M.} (1995).
Markov Processes associated with semi-Dirichlet forms.
\textit{Osaka J. Math}, \textbf{32} 97--119. MR{1323103}

\bibitem[MR92]{MR92}
\textsc{Ma, Z. M.} and \textsc{R\"{o}ckner, M.} (1992).
\textit{Introduction to the theory of (non-symmetric) Dirichlet
  forms}. Springer-Verlag, Berlin. MR{1214375}

\bibitem[MR95]{MR95}
\textsc{Ma, Z. M.} and \textsc{R\"{o}ckner, M.} (1995).
Markov Processes associated with positivity preserving coercive
  forms.
\textit{Cand. J. Math}, \textbf{47} 817--840. MR{1346165}

\bibitem[Os13]{Oshima13}
\textsc{Oshima, Y.} (2013).
\textit{Semi-Dirichlet forms and Markov processes}.
Walter de Gruyter  $\&$ Co., Berlin. MR{3060116}


\bibitem[Sc99]{Sc99} \textsc{Schmulan, B.} (1999). Positivity preserving forms have the Fatou property,  \textit{Potential Anal.} \textbf{10} 373-378. MR{1697255}

\bibitem[Si73]{Si73}
\textsc{Simon, B.} (1973). Ergodic semigroups of positivity preserving self-adjoint operators, \textit{J. Funct. Anal.}. \textbf{12} 335-339. MR{0358434}

\bibitem[Si77]{Si77}
\textsc{Simon, B.} (1977). An Abstract Kato's Inequality for Generators of Positivity Preserving Semigroups, \textit{Indiana Univ. Math. J.}. \textbf{26} 1067-1072. MR{0461209}

\bibitem[Si79]{Si79}
\textsc{Simon, B.} (1979). Kato's inequality and the comparison of semigroups, \textit{J. Funct. Anal.}. \textbf{32} 97-101. MR{0533221}


\bibitem[St64]{Stampacchia64}
\textsc{Stampacchia, G.} (1964).
Formes bilin\'{e}aires coercitives sur les ensembles convexes.
\textit{C. R. Acad. Sci. Paris }, \textbf{258} 4413-4416. MR{0166591}



\end{thebibliography}
\end{document}